\documentclass[11pt,reqno]{amsart}
\usepackage{comment}
\usepackage{graphicx}
\usepackage{tikz}
\usepackage{pifont}
\usepackage{enumitem}
\usepackage{caption}
\setlist[itemize]{leftmargin=2.5em}
\usepackage{mathtools}
\usepackage{amsfonts}
\usepackage{amssymb}

\usepackage[pagebackref]{hyperref}
\hypersetup{
	colorlinks=true,
	allcolors=blue
}
\usepackage[capitalize]{cleveref}

\usepackage[margin=0.8in]{geometry}
\usepackage{setspace}
\newtheorem{theorem}{Theorem}[section]
\newtheorem{lemma}[theorem]{Lemma}
\newtheorem{claim}{Claim}[theorem]
\Crefname{claim}{Claim}{Claims} 
\newtheorem{conjecture}[theorem]{Conjecture}
\Crefname{conjecture}{Conjecture}{Conjectures}

\expandafter\let\expandafter\oldproof\csname\string\proof\endcsname
\let\oldendproof\endproof
\renewenvironment{proof}[1][\proofname]{%
	\oldproof[\normalfont\bfseries #1]%
}{\oldendproof}
\newenvironment{subproof}[1][\normalfont\it\proofname]{%
	\begin{proof}[#1]%
	}{%
	\end{proof}%
}
\newcommand{\cupcup}{\cup\cdots\cup}

\newcommand{\dd}{\textquotedblleft}
\newcommand{\ee}{\textquotedblright}

\newcommand{\mac}{\mathcal}

\newcommand{\vep}{\varepsilon}
\newcommand{\vare}{\varepsilon}

\newcommand{\nin}{\notin}

\newcommand{\ind}{\mathrm{ind}}

\renewcommand{\subset}{\subseteq}
\renewcommand{\supset}{\supseteq}
\newcommand{\erh}{Erd\H{o}s-Hajnal}

\newcommand{\poly}{\operatorname{poly}}
\DeclarePairedDelimiter\abs{\lvert}{\rvert}%
\DeclarePairedDelimiter\ceil{\lceil}{\rceil}%
\DeclarePairedDelimiter\floor{\lfloor}{\rfloor}%
\makeatletter
\newcommand{\leqnomode}{\tagsleft@true}
\newcommand{\reqnomode}{\tagsleft@false}
\newcommand{\nat}{^{\natural}}
\def\DD{\hbox{-}}

\def\LL{,\ldots,}
\makeatother

\begin{document}
	\title[New graphs with the Erd\H os-Hajnal property]{Induced subgraph density. IV. New graphs with the Erd\H os-Hajnal property}
	\author{Tung Nguyen}
	\address{Princeton University, Princeton, NJ 08544, USA}
	\curraddr{Mathematical Institute and Christ Church, University of Oxford, UK}
	\email{\href{mailto:tunghn@math.princeton.edu}
		{tunghn@math.princeton.edu}}
	\author{Alex Scott}
	\address{Mathematical Institute,
		University of Oxford,
		Oxford OX2 6GG, UK}
	\email{\href{mailto:alexander.scott@maths.ox.ac.uk}{alexander.scott@maths.ox.ac.uk}}
	\author{Paul Seymour}
	\address{Princeton University, Princeton, NJ 08544, USA}
	\email{\href{mailto:pds@math.princeton.edu}{pds@math.princeton.edu}}
	\thanks{The first author was supported by AFOSR grant FA9550-22-1-0234, NSF grant DMS-2154169, a Porter Ogden Jacobus Fellowship, a Titchmarsh Research Fellowship, and a Christ Church Research Centre Grant.
		The second author was supported by EPSRC grant EP/X013642/1.
		The third author was supported by AFOSR grant FA9550-22-1-0234 and NSF grant DMS-2154169.}
	\begin{abstract}
		Erd\H os and Hajnal conjectured that for every graph $H$, there exists $c>0$ such that every $H$-free graph $G$ 
		has a clique or a stable set of size at least $|G|^c$
		(a graph is $H$-free if it has no induced subgraph isomorphic to $H$).
		Alon, Pach, and Solymosi reduced the Erd\H os-Hajnal conjecture to the case when $H$ is {\em prime} (that is, $H$ cannot be obtained  
		by vertex-substitution from smaller graphs);
		but until now, it was not shown for any prime graph with more than five vertices.
		
		We will provide infinitely many prime graphs that satisfy the conjecture.
		Let $H$ be a graph with the property that for every prime induced subgraph $G'$ with $|G'|\ge 3$, $G'$ has a vertex of 
		degree one and a vertex of degree $|G'|-2$.
		We will prove that every graph $H$ with this property satisfies the  
		Erd\H os-Hajnal conjecture, and infinitely many 
		graphs with this property are prime.
		More generally, say a graph is {\em buildable} if every prime induced subgraph with at least three vertices has a 
		vertex of degree one.
		We prove that if $H_1$ and $\overline{H_2}$ are buildable, there exists $c>0$ such that 
		every  graph $G$ that is both $H_1$-free and $H_2$-free has a clique or a stable set of size at least $|G|^c$. 
		
		Our proof uses a new technique of ``iterative sparsification'', where we pass to a sequence of successively more restricted induced subgraphs.
		This approach also extends to ordered graphs and to tournaments.  For ordered graphs, we obtain a theorem which significantly extends a recent result
		of Pach and Tomon about excluding monotone paths; and for tournaments, we  obtain infinitely many new prime tournaments that satisfy the Erd\H{o}s-Hajnal conjecture (in tournament form).
		
		

	\end{abstract}
	\maketitle
	\section{Introduction}
	\label{sec:intro}
	All graphs in this paper are finite and simple.
	For a graph $G$, $\overline G$ denotes its complement, and $\abs G:=\abs{V(G)}$.
	An {\em induced subgraph} of $G$ is a graph obtained from $G$ by removing vertices.
	For a graph $H$, a {\em copy} of $H$ in $G$ is an isomorphism from $H$ to an induced subgraph of $G$;
	and $G$ is {\em $H$-free} if there is no copy of $H$ in~$G$.

	A conjecture of Erd\H{o}s and Hajnal~\cite{MR1031262,MR599767} asserts that:
	\begin{conjecture}
		\label{conj:eh}
		For every graph $H$, there exists $c>0$ such that in every $H$-free graph $G$ there is a clique or a stable set of size at least $|G|^c$.
	\end{conjecture}
	
	A graph $H$ satisfying this conjecture is said to have the {\em \erh{} property}.
	Despite considerable attention (see, for example,~\cite{MR3150572,MR1425208} for surveys), the Erd\H os-Hajnal conjecture 
	remains open.  Erd\H os and Hajnal \cite{MR1031262} proved that $H$-free graphs contain cliques or stable sets of size $\exp(c_H\sqrt{\log n})$, and recently, this was improved to $\exp(c_H\sqrt{\log n\log\log n})$ \cite{density1}. But no better general bound is known.
	
	An important result of Alon, Pach, and Solymosi shows that 
	the class of graphs $H$ satisfying \cref{conj:eh} has a certain closure property.
	Given two graphs $H_1,H_2$ and a vertex $v\in V(H_1)$, the graph obtained {\em from $H_1$ by substituting $H_2$ 
		for $v$} is formed by taking the disjoint union of $H_1\setminus \{v\}$ and $H_2$, and then adding edges to make every 
	vertex of $H_2$ adjacent to all the neighbours of $v$ in $H_1$. This operation is {\em vertex-substitution}.  Alon, Pach, and Solymosi~\cite{MR1832443} showed the following:
	\begin{theorem}
		\label{thm:aps}
		Let $H_1,H_2$ have the \erh{} property,
		and let $H$ be obtained from $H_1$ by substituting $H_2$ for a vertex of $H_1$.
		Then $H$ has the \erh{} property.
	\end{theorem}
	A graph $H$ is {\em prime}
	if it cannot be constructed by vertex-substitution from two graphs both with fewer vertices.
	Equivalently, $H$ is prime if there is no $S\subset V(H)$ with $1<\abs S<\abs H$ such that all vertices in $S$ have 
	the same neighbourhood in $V(H)\setminus S$. All graphs with at most two vertices are prime;
	let us say a graph is {\em non-trivial} if it has at least three vertices. 
	
	In view of \cref{thm:aps}, it is enough to show that every prime graph satisfies the Erd\H{o}s-Hajnal conjecture.  However, until now the only non-trivial prime graphs known to have the \erh{} property were the 
	the four-vertex path $P_4$, the
	bull
	(obtained from $P_4$ by adding a new vertex adjacent its two middle vertices)~\cite{MR2462320},
	and the five-vertex cycle $C_5$~\cite{MR4563865}. 
	The five-vertex path $P_5$ is in process of being added to this list~\cite{density7}
	(and see also~\cite{bb2022, density5} for some other recent progress).
	
	In this paper, we will give infinitely many prime graphs that all have the \erh{} property (for example, the graphs of 
	\cref{fig:smallest}; the lower half of \cref{fig:construction} shows a more complicated example).
	The simplest version of our main result is the
	following (we prove stronger versions later):
	\begin{theorem}\label{mainthm0}
		Let $H$ be a graph, such that every prime induced subgraph $H'$ with at least three vertices has both a vertex of degree one 
		and a vertex of degree $|H'|-2$ (and so degree one in the complement). Then $H$ has the \erh{} property.
	\end{theorem} 
	This is remarkable, because until now, and despite 
	intensive effort, there were no prime graphs with more than five vertices that were known to have the \erh{} property.

	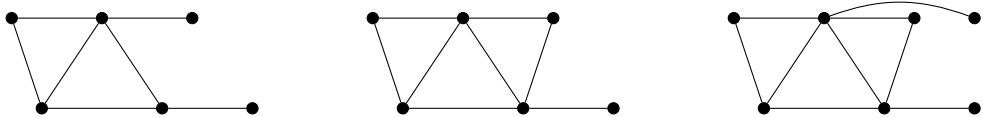
\begin{figure}[h]
		\centering
		
		\tikzstyle{every node}=[inner sep=1.5pt, fill=black,circle,draw]
		\begin{tikzpicture}[scale=0.75,auto=left]
			\node (a1) at (-4.5,1.5) {};
			\node (a2) at (-4,0) {};
			\node (a3) at (-2,0) {};
			\node (a4) at (-.5,0) {};
			\node (a5) at (-3,1.5) {};
			\node (a6) at (-1.5,1.5) {};
			\foreach \from/\to in {a1/a2, a2/a3,a3/a4,a1/a5,a2/a5,a3/a5, a5/a6}
			\draw [-] (\from) -- (\to);
			
			\tikzstyle{every node}=[inner sep=1.5pt, fill=black,circle,draw]
			\node (a1) at (1.5,1.5) {};
			\node (a2) at (2,0) {};
			\node (a3) at (4,0) {};
			\node (a4) at (4.5,1.5) {};
			\node (a5) at (3,1.5) {};
			\node (a6) at (5.5,0) {};
			\foreach \from/\to in {a1/a2, a2/a3,a3/a4,a1/a5,a2/a5,a3/a5,a4/a5, a3/a6}
			\draw [-] (\from) -- (\to);
			
			\tikzstyle{every node}=[inner sep=1.5pt, fill=black,circle,draw]
			\node (a1) at (7.5,1.5) {};
			\node (a2) at (8,0) {};
			\node (a3) at (10,0) {};
			\node (a4) at (10.5,1.5) {};
			\node (a5) at (9,1.5) {};
			\node (a6) at (11.5,0) {};
			\node (a7) at (11.5,1.5) {};
			\foreach \from/\to in {a1/a2, a2/a3,a3/a4,a1/a5,a2/a5,a3/a5,a4/a5, a3/a6}
			\draw [-] (\from) -- (\to);
			\draw [-] (a5) to [bend left=20] (a7);
		\end{tikzpicture}
		\caption{
			The two six-vertex prime graphs in $\mac H$, and one on seven vertices.
		} \label{fig:smallest}
	\end{figure}

	\begin{figure}[h]
		\centering
		
		\begin{tikzpicture}[scale=0.8,auto=left]
			\tikzstyle{every node}=[inner sep=1.5pt, fill=black,circle,draw]
			\node (b0) at (-9,3) {};
			\node (a2) at (-7.5,0) {};
			\node (b3) at (-5.5,3) {};
			\node (b4) at (-6.5,3) {};
			
			\node (a5) at (-4.5,0) {};
			\node (a6) at (-3.5,0) {};
			\node (b7) at (-1.5,3) {};
			
			\node (b8) at (-2.5,3) {};
			\node (a9) at (-.5,0) {};
			\node (a10) at (.5,0) {};
			\node (b11) at (2.5,3) {};
			\node (b12) at (1.5,3) {};
			\node (a13) at (3,0) {};

			\foreach \from/\to in {a2/b3, a2/b4,b4/a5,b4/a6,a6/b7,a6/b8,b8/a9,b8/a10,a10/b11,a10/b12,b12/a13}
			\draw [very thick] (\from) -- (\to);
			
			\tikzstyle{every node}=[]
			\draw[above] (b0) node            {$b_1$};
			\draw[above] (b3) node            {$b_4$};
			\draw[above] (b4) node            {$b_3$};
			\draw[above] (b7) node            {$b_8$};
			\draw[above] (b8) node            {$b_7$};
			\draw[above] (b11) node            {$b_{12}$};
			\draw[above] (b12) node            {$b_{11}$};
			\draw[below] (a2) node            {$a_2$};
			\draw[below] (a5) node            {$a_5$};
			\draw[below] (a6) node            {$a_6$};
			\draw[below] (a9) node            {$a_{9}$};
			\draw[below] (a10) node            {$a_{10}$};
			\draw[below] (a13) node            {$a_{13}$};

			%
			%
			
			\tikzstyle{every node}=[inner sep=1.5pt, fill=black,circle,draw]
			\node (b0) at (-9,-1) {};
			\node (a2) at (-7.5,-4) {};
			\node (b3) at (-5.5,-1) {};
			\node (b4) at (-6.5,-1) {};
			
			\node (a5) at (-4.5,-4) {};
			\node (a6) at (-3.5,-4) {};
			\node (b7) at (-1.5,-1) {};
			
			\node (b8) at (-2.5,-1) {};
			\node (a9) at (-.5,-4) {};
			\node (a10) at (.5,-4) {};
			\node (b11) at (2.5,-1) {};
			\node (b12) at (1.5,-1) {};
			\node (a13) at (3,-4) {};

			\foreach \from/\to in {a2/b3, a2/b4,b4/a5,b4/a6,a6/b7,a6/b8,b8/a9,b8/a10,a10/b11,a10/b12,b12/a13}
			\draw [very thick] (\from) -- (\to);
			\foreach \from/\to in {a2/a5,a5/a6,a6/a9,a9/a10, a10/a13}
			\draw [-] (\from) to [bend right=5] (\to);
			\foreach \from/\to in {a2/a6,a5/a9,a6/a10,a9/a13}
			\draw (\from) to [bend right=9] (\to);
			\foreach \from/\to in {a2/a9,a5/a10,a6/a13}
			\draw (\from) to [bend right=13] (\to);
			\foreach \from/\to in {a2/a10, a5/a13}
			\draw (\from) to [bend right=17] (\to);
			\foreach \from/\to in {a2/a13}
			\draw (\from) to [bend right=21] (\to);
			
			\foreach \from/\to in {b0/a5,b0/a6,b0/a9,b0/a10,b0/a13,b3/a9,b3/a10,b3/a13,
				b4/a9,b4/a10,b4/a13,b7/a13,b8/a13}
			\draw [ultra thin] (\from) -- (\to);

		\end{tikzpicture}
		\caption{
			Start with a path ($a_2\DD b_3\DD a_6 \DD b_7\DD a_{10}\DD b_{11}$ in this case), add a leaf at every vertex, add an isolated 
			vertex $b_1$, 
			and take a bipartition $(A,B)$, numbered as shown.
			Now make $A$ a clique; and make $a_i, b_j$ adjacent if $i\ge j+4$.} \label{fig:construction}
	\end{figure}
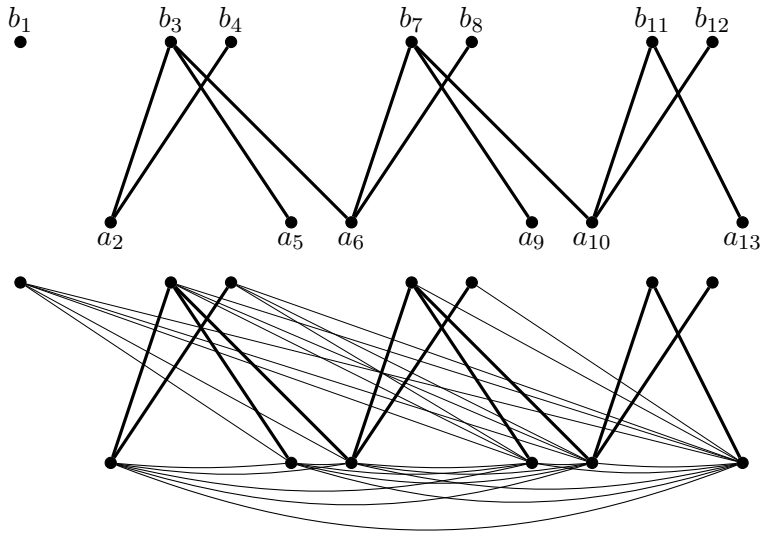
	We define $\mac H$ to be the class of all graphs that can be constructed by a sequence of the following operations, starting with one-vertex graphs:
	\begin{itemize}
		\item choose a graph $G$ that is already constructed, choose a vertex $v\in V(G)$ that has degree at least $|G|-2$,
		and add a new vertex adjacent only to $v$;
		\item choose a graph $G$ that is already constructed, choose a vertex $v\in V(G)$ that has degree at most one,
		and add a new vertex with neighbour set $V(G)\setminus \{v\}$ (this is the operation of the first bullet, in the complement graph);
		\item choose two graphs $H_1,H_2$ that are already constructed, and substitute $H_2$ for a vertex of $H_1$.
	\end{itemize}
	For instance, the graph constructed in \cref{fig:construction} belongs to $\mac H$. (Add one vertex at a time in the order 
	of the numbers.)
	Then:
	\begin{itemize}
		\item the bull and $P_4$ belong to $\mac H$;
		\item  if $H\in \mac H$, then so is its complement and all its induced subgraphs;
		\item $H\in \mac H$ if and only if every nontrivial prime induced subgraph $H'$ of $H$ has a vertex of degree one and a vertex of degree $|H'|-2$; and
		\item each prime graph $H\in\mac H$ is a {\em split graph}, that is,
		its vertex set is the union of a clique and a stable set.
	\end{itemize}
	(The first three bullets are direct consequences of the construction of $\mac H$, and the fourth bullet is the statement of \cref{lem:prime}.)
	By the definition of $\mac H$, \cref{mainthm0} says:
	\begin{theorem}
		\label{thm:eh}
		Every $H\in\mac H$ has the \erh{} property.
	\end{theorem}
	We claim that $\mac H$ contains an infinite number of prime graphs (including the bull, but not $C_5$, and 
	unfortunately not $P_5$).
	Indeed, $\mac H$ contains a prime graph with $h$ vertices for every $h\ge 4$. To see this, observe first 
	that there are prime graphs in $\mac H$ with four vertices and five vertices
	($P_4$ and the bull).
	Second, if $H\in\mac H$ is prime, then there is a prime graph in $\mac H$ with $|H|+2$ vertices; because let $v$ be a vertex
	of $H$ with degree one, adjacent to $u$ say. Add two new vertices,
	one adjacent to all vertices of $H$, and the other just adjacent to $u$; then this enlarged graph is also prime and belongs to 
	$\mac H$. 
	
	We point out that the third bullet in the definition of $\mac H$ is not really important. If we just want to construct all the prime graphs in $\mac H$, the 
	first two bullets are enough. Note, however, that having used, say, the first bullet operation 
	on some vertex $v$, adding a new vertex $u$,
	one can then use the first bullet operation
	again on the same vertex $v$, adding a ``nonadjacent twin'' of $u$: 
	this is the same as substituting a two-vertex stable set for $u$. At that stage, the graph is not prime, but might still 
	eventually grow into a prime graph, because later steps in the growing process might restore primeness. Moreover, if we want to 
	avoid using the third bullet, then 
	this repetition of the same operation on the same vertex may be necessary (for instance, to grow the graph of 
	\cref{fig:construction}). When we come to the ``pairs of graphs'' extension of the result (Theorem \ref{thm:pairsEH}), vertex-substitution will become
	important, and it is convenient to retain it here to make that extension simpler.

	The \erh{} property of the bull was first proved by Chudnovsky and Safra~\cite{MR2462320} using the strong perfect graph theorem~\cite{MR2233847} and a decomposition theorem for bull-free graphs,
	and later reproved by Chudnovsky, Scott, Seymour, and Spirkl~\cite{MR4563865} via a different approach that simultaneously
	showed the \erh{} property of $C_5$.
	Our proof of \cref{thm:eh} gives a third proof of the \erh{} property of the bull.
	
	The result of this paper gives two prime six-vertex graphs that have the \erh{} property.
	We have been striving, for the last forty years or so (some of us, anyway) to prove that all five-vertex graphs have the 
	\erh{} property, and we have just succeeded~\cite{density7}. 
	Where are we on
	six-vertex graphs? There are ten prime six-vertex graphs that contain $P_5$ or its complement, and six that do not.
	The result of this paper handles two of the six, those in \cref{fig:smallest}. The other four, two complementary pairs, are
	shown in \cref{fig:6vertex}; and at least they have a ``near-Erd\H{o}s-Hajnal'' property~\cite[Section 8.4]{2025thes}.
	The ten that contain $P_5$ or its complement are even more challenging, since $P_5$ itself is very difficult;
	but assuming the result for $P_5$, we can handle two of the ten, using the results of this paper. We discuss this more in 
	\cref{sec:cor}.
	
	\begin{figure}[h]
		\centering
		
		\tikzstyle{every node}=[inner sep=1.5pt, fill=black,circle,draw]
		\begin{tikzpicture}[scale=0.75,auto=left]
			\node (a1) at (-2,0) {};
			\node (a2) at (-.5,0) {};
			\node (a3) at (1,0) {};
			\node (a4) at (-2,1.5) {};
			\node (a5) at (-0.5,1.5) {};
			\node (a6) at (1,1.5) {};
			
			\node (b1) at (2,0) {};
			\node (b2) at (4,0) {};
			\node (b3) at (6,0) {};
			\node (b4) at (3,1.5) {};
			\node (b5) at (5,1.5) {};
			\node (b6) at (7,1.5) {};
			
			\node (c1) at (8.4,.5) {};
			\node (c2) at (9,1.5) {};
			\node (c3) at (9.6,.5) {};
			\node (c4) at (7.6,0) {};
			\node (c5) at (9,2.4) {};
			\node (c6) at (10.4,0) {};
			
			\node (d1) at (11.4,0) {};
			\node (d2) at (13,0) {};
			\node (d3) at (14.6,0) {};
			\node (d4) at (12.2,1.3) {};
			\node (d5) at (13.8,1.3) {};
			\node (d6) at (13,2.6) {};

			\foreach \from/\to in {a1/a2, a2/a3,a4/a5,a5/a6,a2/a5,a3/a6,b1/b2,b2/b3,b4/b5,b5/b6,b1/b4,b4/b2,b2/b5,b5/b3,b3/b6, c1/c2,c1/c3,c2/c3,c1/c4,c2/c5,c3/c6, d1/d2,d2/d3,d1/d4,d2/d4,d2/d5,d3/d5,d4/d5,d4/d6,d5/d6}
			\draw [-] (\from) -- (\to);
		\end{tikzpicture}
		\caption{
			The six-vertex graphs not containing $P_5$ or $\overline{P_5}$ that remain open.
		} \label{fig:6vertex}
	\end{figure}
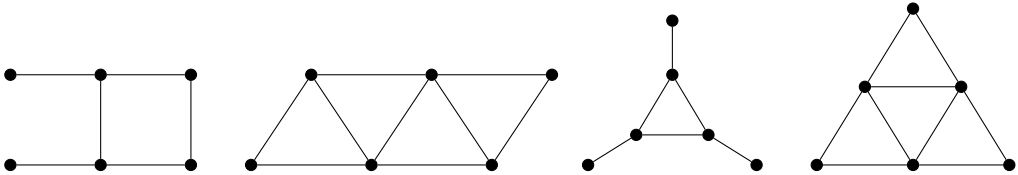
	
	We will actually prove a statement more general than \cref{thm:eh} (in fact, our proof method {\em requires} most of this additional generality). 
	Our final result is
	more general than \cref{thm:eh} in four ways:
	\begin{itemize}
		\item We will prove that these graphs satisfy a conjecture of Fox and Sudakov (explained below), not just that they have the \erh{} property.
		\item We can obtain the same conclusion under a more relaxed hypothesis that there are not many copies of $H$ in $G$. 
		(By a recent result of
		Buci\'c, Fox and Pham~\cite{bfp2024}, these first two strengthenings are now both known to follow from the fact that $H$ has the Erd\H{o}s-Hajnal property.
		However, we will not use the result of~\cite{bfp2024}.)
		\item Each prime graph in $\mathcal{H}$ has a vertex of degree one and so does its complement, and this is what we need for the inductive proof to work. 
		It is just as good, and gives a stronger theorem, if we exclude two graphs instead of one, with 
		the property
		that every non-trivial prime induced subgraph of the first has a vertex of degree one, and the same for the complement of the second.
		\item All this works just as well for ordered graphs; and we obtain consequences for ordered graphs and tournaments.
	\end{itemize}
	Let us explain these things in more detail.
	
	Throughout the paper we will need to work with graphs that are suitably dense or sparse.
	For $\vep\in(0,\frac12)$, a graph $G$ is {\em $\vep$-restricted} if one of $G,\overline G$ has maximum degree at most $\vep\abs G$.
	We say that $S\subseteq V(G)$ is {\em $\vep$-restricted} in $G$ if $G[S]$ is $\vep$-restricted where $G[S]$ is the subgraph of $G$ induced on $S$.
	An important  theorem of R\"odl~\cite{MR837962} shows that $H$-free graphs contain large $\vep$-restricted subsets for any fixed $\vep>0$:
	\begin{theorem}
		\label{thm:rodl}
		For every graph $H$ and every $\vep\in(0,\frac12)$, there exists $\delta>0$ such that for every $H$-free 
		graph $G$, there is an $\vep$-restricted subset $S\subseteq V(G)$ with size at least $\delta \abs G$.
	\end{theorem} 
	R\"odl's theorem can be strengthened to allow a few copies of $H$.
	If $G,H$ are graphs, $\ind_H(G)$ denotes the number of copies of $H$ in $G$.
	Nikiforov~\cite{MR2271833} extended \cref{thm:rodl} as follows:
	\begin{theorem}
		\label{thm:niki}
		For every graph $H$ and every $\vep\in(0,\frac12)$, there exists $\delta>0$ such that if $G$ is a graph
		with $\ind_H(G)\le (\delta\abs G)^{\abs H}$,
		there is an $\vep$-restricted subset $S\subseteq V(G)$ with size at least $\delta \abs G$.
	\end{theorem}
	The proofs of R\"odl and Nikiforov used the regularity lemma, and gave bounds for
	$\delta^{-1}$ that were tower-type in terms of $\epsilon^{-1}$.
	Fox and Sudakov~\cite{MR2455625} gave much better bounds, using a different proof method; they proved that, in both theorems, $\delta$ can be chosen to be $2^{-d(\log\frac1\vep)^2}$
	for some $d>0$ depending only on $H$.
	They also conjectured the following ``polynomial R\"odl'' version of \cref{thm:rodl}:
	\begin{conjecture}
		\label{conj:polyrodl}
		For every graph $H$, there exists $d>0$ such that for every $\vep\in(0,\frac12)$,
		and every $H$-free graph $G$, there is an $\vep$-restricted subset of $V(G)$ with size at least $\vep^d\abs G$.
	\end{conjecture}
	Taking a fixed value of $\vep$ implies \cref{thm:rodl}.  However, the polynomial dependence allows us to take much smaller $\vep$: in particular, by taking $\vep$ to be a small negative power of $n$, it follows that \cref{conj:polyrodl}
	implies the Erd\H{o}s-Hajnal conjecture.
	Also, the following ``polynomial Nikiforov'' statement unifies \cref{thm:niki} and \cref{conj:polyrodl}.
	\begin{conjecture}
		\label{conj:viral}
		For every graph $H$, there exists $d>0$ such that for every $\vep\in(0,\frac12)$ and every graph $G$ 
		with $\ind_H(G)\le (\vep^d\abs G)^{\abs H}$, there is an $\vep$-restricted subset of $V(G)$ with size at least $\vep^d\abs G$.
	\end{conjecture}
	Let us say a graph $H$ is {\em viral} if it satisfies \cref{conj:viral}. Thus, viral graphs satisfy \cref{conj:polyrodl}; and it was recently shown by Buci\'c, Fox, and Pham~\cite{bfp2024} that a graph $H$ has the \erh{} property if and only if it is viral.
	Recent~developments on the Erd\H{o}s-Hajnal conjecture suggest that being viral could be the ``right'' concept to investigate due to its ``counting induced subgraphs'' idea.
	Indeed, this idea was crucial behind the proof of the best known general bound $e^{c\sqrt{\log n\log\log n}}$ (for some $c$ depending on $H$) towards \cref{conj:eh}, which was actually done by directly proving the best known
	quantitative dependence of $\delta$ on $\vep$ known in \cref{thm:rodl,thm:niki}:
	\[\delta=e^{-d(\log\frac1\vep)^2/\log\log\frac1\vep}\]
	(for some $d>0$ depending on $H$).
	We will prove the following equivalent version of \cref{thm:eh}:
	\begin{theorem}
		\label{thm:viral}
		Every $H\in\mac H$ is viral.
	\end{theorem}
	The proof of this is by induction on $|H|$ and underlines the importance of the viral property.
	There are two reasons why we chose to present a direct proof of \cref{thm:viral} in this paper:
	\begin{itemize}
		\item First, even if we just wanted to prove \cref{thm:eh} that the graphs in $\mac H$ have the \erh{} property, it 
		is essential, for our inductive argument to work, that we have a more general inductive hypothesis that the proper induced subgraphs of $H$ are viral.
		This was actually done in~\cite[Chapter~3]{2025thes} where the main argument is only a little simpler numerically.
		To the best of our knowledge, this paper was also the first time that the viral property was crucial in providing (infinitely many) new graphs with the \erh{}~property.
		
		\item Second, in order to prove the viral property of the graphs in $\mac H$ directly, we developed an idea that turns \dd uniformly dense or sparse\ee{} sequences of disjoint vertex subsets with \dd floating\ee{} parameters into polynomially dense induced subgraphs; and this is \cref{transfer}.
		From our viewpoint, this idea is interesting in its own right because it became the main inspiration behind a key step in the proof of the \erh{} property of $P_5$~\cite{density7}. 
	\end{itemize}
	
	Although progress on the \erh{} property itself has been slow, there are several papers in the literature showing that graphs 
	that contain neither of two given graphs have polynomial-sized cliques or stable sets. For instance, it is shown 
	in~\cite{MR4170220}
	that if $H_1, \overline{H_2}$ are forests, there exists $c>0$ such that every graph $G$ that is both $H_1$-free and $H_2$-free
	has a clique or stable set of size at least $|G|^c$. The reason this ``pair of graphs'' approach is comparatively so successful, 
	is that the proof method uses \cref{thm:rodl} as the first step, and thereafter works inside a subgraph that is either very 
	sparse or very dense. One of the graphs $H_1,H_2$ is good for the sparse case, and the other for the dense case, while it may be
	difficult to find a single graph that is good for both cases simultaneously.
	One could try to derive a graph with the \erh{} property by asking that $H_1=H_2$; but for instance, if $H$
	is both a forest and the complement of a forest, then $H$ has at most four vertices, and we already know that such graphs 
	have the \erh{} property. The same happens for all the ``pair of graphs'' results found so far: if we insist that the same graph
	$H$ fills both roles, we get nothing of interest. But in the present paper, that is not so. There is a ``pair of graphs'' version,
	which is perhaps simpler and more natural; and it remains nontrivial (and gives \cref{thm:viral}) if we insist that the two 
	graphs are the same.
	
	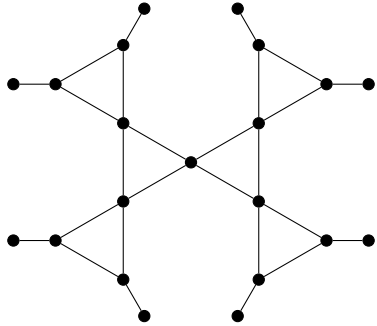
\begin{figure}[h]
		\centering
		
		\tikzstyle{every node}=[inner sep=1.5pt, fill=black,circle,draw]
		\begin{tikzpicture}[scale=0.8,auto=left]
			\def\r{1.3}
			\def\s{.7}
			\node (a) at (0,0) {};
			\node (b1) at ({\r*cos(30)},{\r*sin(30)}) {};
			\node (b2) at ({\r*cos(-30)},{\r*sin(-30)}) {};
			\node (b3) at ({\r*cos(210)},{\r*sin(210)}) {};
			\node (b4) at ({\r*cos(150)},{\r*sin(150)}) {};
			\node (c1) at ({\r*cos(30)},{\r*sin(30)+\r}) {};
			\node (c2) at ({2*\r*cos(30)},{2*\r*sin(30)}) {};
			\node (c3) at ({2*\r*cos(-30)},{2*\r*sin(-30)}) {};
			\node (c4) at ({\r*cos(-30)},{\r*sin(-30)-\r}) {};
			\node (c5) at ({\r*cos(210)},{\r*sin(210)-\r}) {};
			\node (c6) at ({2*\r*cos(210)},{2*\r*sin(210)}) {};
			\node (c7) at ({2*\r*cos(150)},{2*\r*sin(150)}) {};
			\node (c8) at ({\r*cos(150)},{\r*sin(150)+\r}) {};
			\node (d1) at ({\r*cos(30)+\s*cos(120)},{\r*sin(30)+\r+\s*sin(120)}) {};
			\node (d2) at ({2*\r*cos(30)+\s},{2*\r*sin(30)}) {};
			\node (d3) at ({2*\r*cos(-30)+\s},{2*\r*sin(-30)}) {};
			\node (d4) at ({\r*cos(-30)+\s*cos(-120)},{\r*sin(-30)-\r+\s*sin(-120)}) {};
			\node (d5) at ({\r*cos(210)+\s*cos(300)},{\r*sin(210)-\r+\s*sin(300)}) {};
			\node (d6) at ({2*\r*cos(210)-\s},{2*\r*sin(210)}) {};
			\node (d7) at ({2*\r*cos(150)-\s},{2*\r*sin(150)}) {};
			\node (d8) at ({\r*cos(150)+\s*cos(60)},{\r*sin(150)+\r+\s*sin(60)}) {};

			\foreach \from/\to in {a/b1,a/b2,a/b3,a/b4,b1/b2,b3/b4,b1/c1,b1/c2,c1/c2,b2/c3,b2/c4,c3/c4,b3/c5,b3/c6,c5/c6,b4/c7,b4/c8,c7/c8,
				c1/d1,c2/d2,c3/d3,c4/d4,c5/d5,c6/d6,c7/d7,c8/d8}
			\draw [-] (\from) -- (\to);

		\end{tikzpicture}
		\caption{
			With $H$ as shown, $H\in \mac J$, and so $\{H,\overline{H}\}$ is viral.
		} \label{fig:tritree}
	\end{figure}
	
	We define $\mac J$ to be the class of all graphs that can be constructed by a sequence of the following operations, starting with one-vertex graphs:
	\begin{itemize}
		\item choosing a graph $G$ that is already constructed, choosing a vertex $v\in V(G)$,
		and adding a new vertex adjacent only to $v$;
		\item choosing two graphs $H_1,H_2$ that are already constructed, and substituting $H_2$ for a vertex of $H_1$.
	\end{itemize}
	Equivalently,
	$\mac J$ is the family of graphs $J$ with the property that every induced subgraph of $J$ either contains a vertex of
	degree at most one or is not prime.
	For instance, $\mac J$ contains every forest, and all line graphs of forests.
	We prove the following:
	
	\begin{theorem}\label{thm:pairsEH}
		If $H_1,\overline{H_2}\in \mac J$, there exists $c>0$ such that if $G$ is both $H_1$-free and $H_2$-free, then
		$G$ has a clique or stable set of size at least $|G|^c$.
	\end{theorem}
	The graphs $H$ such that $H\in \mac J$ and $\overline{H}\in \mac J$ are precisely the graphs in $\mac H$ (see \cref{lem:char}), so this
	implies that the members of $\mac H$ have the \erh{} property, by taking $H_1=H_2\in \mac H$.  Note also that, as $\mac J$ contains all forests,
	\cref{thm:pairsEH} contains the theorem of~\cite{MR4170220} mentioned earlier, about excluding a forest and a forest complement.
	We will prove \cref{thm:pairsEH} by showing its viral version, 
	\cref{thm:mainunordered}.
	
	The proof method also applies to ordered graphs (an {\em ordered graph} is a graph with a total order on its vertex set). An argument of Alon, Pach and Solymosi~\cite{MR1832443} shows that 
	the Erd\H{o}s-Hajnal conjecture
	is equivalent to the same statement for ordered graphs. One can define ``vertex-substitution'' and ``prime'' for ordered graphs
	just as for graphs, and again it suffices to consider only prime ordered graphs. But the only prime ordered graphs that 
	(until now) we knew had the ordered \erh{} property had at most {\em three} vertices. We will provide infinitely many. Indeed,
	each graph in $\mathcal{H}$ can be ordered to make a prime ordered graph with the ordered \erh{} property. For instance, the graph
	of \cref{fig:construction}, when ordered such that 
	$$b_{12}\le b_{11}\le b_8\le b_7\le b_4\le b_3\le b_1\le a_2\le a_5\le a_6\le a_{9}\le a_{10}\le a_{13}$$
	becomes a prime ordered graph that has the ordered \erh{} property.
	
	We define $\mac K$ to be the class of all ordered graphs that can be constructed by a sequence of the following operations, starting with one-vertex ordered graphs:
	\begin{itemize}
		\item choosing a graph $G$ that is already constructed, and adding a vertex of degree at most one at one end or the other of its linear order;
		\item choosing two ordered graphs $H_1,H_2$ that are already constructed, and substituting $H_2$ for a vertex of $H_1$.
	\end{itemize}
	We will prove:
	\begin{theorem}\label{thm:buildord}
		If $H_1,\overline{H_2}\in \mac K$, there exists $c>0$ such that if $G$ is an ordered graph that is both
		$H_1$-free and $H_2$-free (in the appropriate sense for ordered graphs), then $G$ has a clique or stable set of size at least $|G|^c$.
	\end{theorem}
	This contains both \cref{thm:pairsEH,thm:buildtour}. A recent theorem of Pach and Tomon~\cite{MR4273057} 
	asserts the special case of  \cref{thm:buildord} where 
	$H_1$ and $\overline{H_2}$ are both obtained by giving a path its natural ordering.
	
	All these results will be extended, in \cref{thm:mainorderedpair} and \cref{thm:mainunordered}, to say that the corresponding objects are viral. This extension
	is critical for inductive reasons.
	
	We can apply \cref{thm:buildord} to tournaments, and obtain new tournaments with the \erh{} property. (See~\cite{pure10} for 
	some related results.)
	Say a tournament is {\em buildable} if it can be grown from nothing
	by repeatedly either adding a vertex of out-degree $\le 1$  or in-degree $\le 1$, or vertex-substitution. We will show:
	\begin{theorem}\label{thm:buildtour}
		For every buildable tournament $H$, there exists $c$ such that if $G$ is a tournament with no subtournament isomorphic to $H$, then
		there is a transitive set in $V(G)$ with size at least $|G|^c$.
	\end{theorem}
	
	As in \cite{density1, polyP4}, we work with dense or sparse subsets instead of large cliques or stable sets, and rather than forbidding an induced subgraph $H$ we allow a few copies: these are necessary generalizations for our approach to work.  However, the tools from these earlier papers are not strong enough to obtain the polynomial bounds we need here for Erd\H os-Hajnal.  In this paper we use a new method of ``iterative sparsification'': we start with a linear-sized dense or sparse subset $S_0$ as given by R\"odl's theorem, and then move iteratively through a sequence $S_0\supset S_1 \supset\cdots$ of subsets that are ever denser or sparser, but not too small.  If we can keep repeating this process, and keep control of the trade-off between size and density (in particular, maintaining a polynomial relationship between the two) then we can continue the process until we obtain a polynomial-sized clique or stable set.  Stopping the process at an intermediate point gives us a set that is dense or sparse to a given degree, and proves the viral property (as well as the Fox-Sudakov conjecture) for $H$.  The key to making this strategy work is then finding a method of passing to a  large induced subgraph of a dense or sparse graph and increasing the density or sparsity.  We discuss this in the next section.
	
	\section{Ordered graphs, and a sketch of the proof}\label{sec:sketch}
	
	The main motivation for our work was unordered graphs and the \erh{} conjecture, but the proof works equally well for ordered
	graphs, and we thereby gain a much more powerful result.
	We would like to outline the idea of the proof as soon as we can, but we need first to set up more definitions, particularly
	about ordered graphs.
	
	An {\em ordered graph} is a pair $G=(F,\le)$ where $F$ is a graph and $\le $ 
	is a linear order of $V(F)$; and we define $G\nat:=F$, and we define $\le_G$ to equal $\le $.
	A {\em copy} of an ordered graph $H$ in an 
	ordered graph $G$ is an isomorphism $\phi$ from $H\nat$ to an induced subgraph $J$ of $G\nat$, such that for all 
	distinct $u,v\in V(H)$, $u\le_H v$ if and only if $\phi(u)\le _G \phi(v)$. We extend many definitions for graphs to 
	ordered graphs in the natural way; so for instance, if $G$ is an ordered graph, 
	we write $V(G):=V(G\nat)$; $|G|:=|G\nat|$;
	$\overline{G}:=(\overline{G\nat},\le_G)$;
	``degree in $G$'' means degree in $G\nat$; a ``blockade in $G$'' means a blockade in $G\nat$; and so on.
	We use  $[n]$ to denote $\{1,2,\ldots,n\}$ for every integer $n\ge 1$.
	
	We will need to control the number of copies of particular subgraphs.  For  graphs $G,H$, and $x>0$, define
	\[\mu_H(x,G):=\frac{\ind_H(G)}{(x\abs G)^{\abs H}};\]
	and for a finite set $\mac F$ of graphs, let
	$$\mu_{\mac F}(x,G):=\max_{H\in\mac F}\mu_H(x,G).$$ 
	Thus $\mu_{\mac F}(x,G)\le 1$ if and only if $G$ contains at most $x^{|H|}|G|^{|H|}$ induced copies of $H$ for each $H\in\mac F$.
	
	We say that a finite set $\mac F$ of graphs is {\em viral} if there exists $d>0$ such that for every $\vep\in(0,\frac12)$,
	and for every graph $G$ with $\mu_{\mac F}(\vep^d,G)\le 1$, there is an $\vep$-restricted $S\subseteq V(G)$ with 
	$\abs S\ge \vep^d\abs G$. We call such a number $d$ a {\em viral exponent} for $\mac F$.
	Thus a graph $H$ is viral if and only if $\{H\}$ is viral.
	These definitions extend to ordered graphs in the natural way. Thus, when $G,H$ are ordered graphs,
	$\ind_H(G)$ denotes the number of copies of $H$ in $G$, and so on.
	
	We need to define vertex-substitution for ordered graphs. Let $H_1,H_2$ be ordered graphs, and let
	$v\in V(H_1)$. The ordered graph $H$ obtained {\em from $H_1$ by substituting $H_2$
		for $v$} is the pair $(H\nat,\le_H)$, where $H\nat$ is the graph obtained from $H_1\nat$ by substituting $H_2\nat$
	for $v$, and $\le_H$ is defined by:
	\begin{itemize}
		\item if $x,y\in V(H_1)\setminus \{v\}$ then $x\le_H y$ if and only if $x\le_{H_1} y$;
		\item if $x,y\in V(H_2)$ then $x\le_H y$ if and only if $x\le_{H_2} y$;
		\item if $x\in V(H_1)\setminus \{v\}$ and $y\in V(H_2)$, then $x\le_H y$ if and only if $x\le_{H_1}v$.
	\end{itemize}
	An ordered graph is {\em prime} if it cannot be obtained by vertex-substitution from two smaller ordered graphs.
	
	We say $v\in V(H)$ is the {\em first} vertex of an ordered graph $H$ if $v\le_H u$
	for all $u\in V(H)$, and the {\em last} vertex is defined similarly.
	We say that $v$ is an {\em end vertex} of $H$ if $v$ is either the first or last vertex of $H$.
	Let $\mac K$ be the class of all ordered graphs $K$ with the property that for every
	induced ordered subgraph $G$ of $K$, either $G$ is not prime, or there exists $v\in V(G)$ with degree at most one in $G\nat$,
	such that $v$ is an end vertex of $H$.
	Our ultimate goal is to prove the following:
	\begin{theorem}
		\label{thm:mainorderedpair}
		For all $H,J\in\mac K$, the pair $\{H,\overline J\}$ is viral.
	\end{theorem}
	All the other theorems we mentioned will be corollaries of this.
	
	Before we can sketch the proof, we need a few more definitions. 
	For a graph $G$ and disjoint $A,B\subseteq V(G)$, $B$ is {\em $x$-sparse} to $A$ if every vertex in $B$ has at most $x\abs A$ neighbours in $A$,
	and {\em $(1-x)$-dense} to $A$ if every vertex in $B$ has at most $x\abs A$ nonneighbours in $A$.
	A {\em blockade} in a graph or ordered graph $G$ is 
	a finite sequence $(B_1,\ldots,B_n)$ of (possibly empty) disjoint subsets of $V(G)$;
	its {\em length} is $n$ and its {\em width} is $\min_{i\in[n]}\abs{B_i}$.
	We also say that $B_1,\ldots,B_n$ are the {\em blocks} of this blockade.
	For $k,w\ge0$, $(B_1,\ldots,B_n)$ is a {\em $(k,w)$-blockade} if its length is at least $k$ and its width is at least $w$.
	For $x\in(0,\frac12)$, this blockade is {\em $x$-sparse} if $B_j$ is $x$-sparse to $B_i$ for all $i,j\in[n]$ with $i<j$,
	and {\em $(1-x)$-dense} if $B_j$ is $(1-x)$-dense to $B_i$ for all $i,j\in[n]$ with $i<j$.  A central part of our argument will involve finding bloackades that are both large and satisfy suitable sparsity conditions.

	Here, finally, is a sketch of the proof of \cref{thm:mainorderedpair}. 
	Let us note that in the actual proof the variables we use will be slightly different from those in the sketch.

	\noindent{\em Proof sketch of \cref{thm:mainorderedpair}.}
	We work by induction on $|H|+|J|$. 
	If one of $H,J$ is not prime, 
	the result follows easily
	from the inductive hypothesis,  by means of \cref{lem:sub} below.
	So we can assume they are both prime, and hence, for each of $H,J$, some end
	vertex has degree one. Let $H',J'$ be obtained from $H,J$ respectively, by deleting an end vertex of degree one.
	Inductively we know that $\{H',\overline{J}\}$ and $\{H,\overline{J'}\}$ are both viral; and this will be enough to imply that
	$\{H,\overline J\}$ is viral (\cref{axioms}). Let $d_0$ be large enough to be a viral exponent for both $\{H',\overline{J}\}$ and $\{H,\overline{J'}\}$.

	There is a key result, \cref{thm:divisive} below, that deduces the property of being viral from the
	existence of suitable blockades. It follows that
	if there exists $b$ such that for every ordered graph $G$ with $\mu_{\mac F}(x^b,G)\le 1$, 
	and every $x\in(0,\frac12)$, there
	is an $x$-sparse or $(1-x)$-dense blockade in $G$ with appropriate length and width, then $\mathcal{F}$ is viral.
	Because of this, we will try to find such a blockade, instead of trying to find directly a large $\vare$-restricted set.
	Thus, let $b\ge d_0$ be some large number,  let $x\in(0,\frac12)$, and let
	$G$ be an ordered graph, with $\mu_{\mac F}(x^b,G)\le 1$. We assume for a contradiction that the blockade we want does not exist; 
	that is, there is no $x$-sparse or $(1-x)$-dense blockade in $G$ with length $k$, where $2\le k\le 1/x$, and width at least 
	$\lfloor|G|/k^b\rfloor$.
	
	Let $h=\max(|H|,|J|)$ and $d\ge d_0$ be some appropriately large number (and later we will choose $b$ to be about $hd^2$).
	We will grow a nested sequence of subsets $V(G)=S_0\supseteq S_1\supseteq \cdots$, where for each $i$,
	$\abs{S_i}\ge 2^{-6hd^{i}}\abs{S_{i-1}}$ and $G[S_i]$ is $2^{-4hd^{i-1}}$-restricted. We have a method to define $S_{i+1}$ 
	in terms of $S_i$, provided $S_i$ is $2^{-4hd^{i-1}}$-restricted;  but it does not work to 
	define $S_1$, because $G$ is generally not $c$-restricted for any fixed $c>0$, so 
	we have to do something else to get $S_1$. We need a large subset $S_1$ which is $2^{-4h}$-restricted.
	Nikiforov's theorem would give us such a thing, except we are working with ordered graphs; so we need an ordered graphs
	version of Nikiforov's theorem (\cref{thm:orderedniki} below). This gives $G[S_1]$. The latter will either have small 
	maximum degree, or the same in the complement, and by moving to the complement if necessary we may assume it has small maximum degree.
	
	Now we will describe
	the general step, to obtain $S_{i+1}$ from $S_i$ when $i\ge 1$.
	We recall that $H'=H\setminus \{v\}$, where $v$ is an end vertex of $H$ with degree one.
	Now we will use that $\{H',\overline{J}\}$ is viral.
	(In the other case, when $G[S_1]$ has small maximum degree in the 
	complement, we would use that $\{H,\overline{J'}\}$ is viral.) Reversing the order if necessary, we may assume that 
	$v$ is the last vertex of $H$.

	Let $y=2^{-2hd^{i-1}}$. Thus $S_i$ is  $y^2$-restricted.
	$G$ is an ordered graph; let $S$ be the set of the first $\lceil y|S_i|\rceil$ vertices of $S_i$. If $G[S]$ does not contain 
	many copies of $H'$ (and we already know that it does not contain many copies of $\overline{J}$, since $G$ itself does not), 
	we can use
	that $\{H',\overline{J}\}$ is viral to find $S_{i+1}\subseteq S$. So we assume that it does contain many copies of $H'$. Let
	$u$ be the neighbour of $v$ in $H$, and let $H''=H\setminus \{u,v\}$. It follows that $G[S]$ contains many copies of $H''$
	that each can be extended to many copies of $H'$ in $G[S]$. But they cannot all be extended to many copies of $H$ in $G$,
	because there are not that many copies of $H$. So there is a copy $X$ of $H''$ in $G[S]$, that extends to many copies of $H'$ in $G[S]$,
	and yet does not extend to many copies of $H$ in $G$; see \cref{fig:sketch}.
	
	\begin{figure}
		\begin{tikzpicture}[scale=0.7]
			\draw[rounded corners] (-12,3) -- (-12,-3) -- (-7,-3) -- (-7,3) -- cycle;
			\draw (-9.5,3.5) node [] {$S$: $\abs S=\ceil{y\abs{S_i}}$};
			\draw[->] (-11.6,3.5) to (-12,3.5);
			\draw[->] (-7.4,3.5) to (-7,3.5);
			\filldraw[color=black,color=black!10,fill=black!10] (-10.5,0.63) -- (-4.4,-1) -- (-4.4,3) -- (-9.8,2.49);
			\filldraw[color=black,fill=white] (-9.5,1.5) ellipse (2cm and 1cm);
			\draw (-9.5,1.5) node [] {$X\cong H\setminus\{u,v\}$};
			\draw (-9.5,-2.5) node[align=center] {$\abs{A_1}\ge \poly(y)\abs{S_i}$};
			\draw[rounded corners] (0,3) -- (0,-1) -- (-4.5,-1) -- (-4.5,-3) -- (8,-3) -- (8,3) -- cycle;
			\draw[rounded corners,color=green,fill=green!20] (-4.5,-1) -- (-4.5,3) -- (0,3) -- (0,-1) -- cycle;
			\draw (-2.25,1) node [] {$\le(\abs H-2)y^2\abs{S_i}$};
			\draw[rounded corners,color=red,fill=red!20] (0,-1) -- (-4.5,-1) -- (-4.5,-3) -- (0,-3) -- cycle;
			\draw (-2.25,-2) node [] {$\le y^2\abs {S_i}$};
			\draw (4,0) node [] {$\ge (1-\abs Hy)\abs {S_i}$};
			\node[inner sep=1pt, fill=black,circle,draw] (w) at (-4.25,-2) {};
			\filldraw[color=black,fill=black!10] (-9,-1.25) -- (w) -- (-9,-2) -- (-9,-1.25);
			\draw[rounded corners,color=blue,fill=blue!20] (-11.25,-2) -- (-11.25,-0) -- (-7.75,-0) -- (-7.75, -2) -- cycle;
			\draw (-9.5,-1) node [align=center] {$A_1$};
			\draw (-5.75,-2.5) node [align=center] {$\ge x\abs{A_1}$};
			\draw (4,3.5) node[] {$B_1$: $x$-sparse to $A_1$};
			\draw[->] (1.5,3.5) to (0,3.5);
			\draw[->] (6.5,3.5) to (8,3.5);
		\end{tikzpicture}
		\caption{Inside $G[S_i]$. The blue box represents $A_1$, the set of choices for $u$; the green box represents the set of vertices with a neighbour in $V(X)$; and the red box represents the set of vertices not in the green box and with at least $x\abs {A_1}$ neighbours in $A_1$.
		} \label{fig:sketch}
	\end{figure}
	
	Let $A_1$ be the set of vertices in $S$ that give an extension of $X$ to a copy
	of $H'$; so $A_1$ is large, all the vertices in $A_1$ have the same neighbours in $V(X)$, and they are all in the right position in the order $\le_G$
	with respect to the vertices in $X$. (For simplicity, we are conflating the copy $X$, which is an isomorphism,
	with its image, an induced subgraph of $G[S]$.) Since $G[S_i]$ has maximum degree at most $y^2|S_i|$, there are not many vertices 
	in the rest of $S_i$ that have a neighbour in $V(X)$; and for the others, say $B$, any edge between $A_1$ and $B$ gives a copy
	of $H$. (This is where we use that $v$ is the last vertex of $H$ and $S$ is an initial segment of $G[S_i]$; 
	all the vertices in $B$ are in the right order relative to $X$ and $A_1$.) Since $X$ does not extend to many copies of $H$, 
	there are not many edges between $A_1$ and $B$; and by throwing away a 
	few outliers, we can choose $B_1\subseteq B$, $x$-sparse to $A_1$, where $B_1$ still contains almost all of $S_i$. We have produced
	an $x$-sparse blockade $(A_1,B_1)$ in $G[S_i]$ of length two, where $|A_1|$ is at least $\poly(y)|G|$, and $|B_1|$ is only slightly smaller than $|S_i|$.
	That is the only argument where we use the leaf of $H$; it is \cref{lem:sparse1}.
	
	Now we look inside the set $B_1$ above, and repeat the same argument; we obtain either the desired set $S_{i+1}$, or
	an $x$-sparse blockade $(A_1,A_2,B_2)$ in $G[S_i]$ 
	of length three, where $|A_1|,|A_2|$ are both at least $\poly(y)|G|$, and $|B_2|$ is only slightly smaller than
	$|S_i|$. By repeating this $1/y$ times (which we can, there is enough room), we obtain either the set $S_{i+1}$ that we want, 
	or the blockade that we assumed did not exist. This is \cref{lem:sparse2}.

	\section{Some lemmas about counting subgraphs}

	In what follows,
	for sets $X,Y,Z$ with $Z\subseteq X$ and a map $f\colon X\to Y$, let $f\vert_Z$ denote the restriction of $f$ to $Z$.
	If $H$ is an ordered graph, and $X\subseteq V(H)$,
	$H\setminus X$ and $H[X]$ are defined in the natural way.

	We need to extend Nikiforov's theorem (\cref{thm:niki}) to ordered graphs, and to do so, we use the following result of R\"odl and Winkler~\cite{rodlwinkler}:
	\begin{theorem}\label{rodlwinkler}
		For every ordered graph $J$ there is a graph $H$ such that, with every ordering of $V(H)$, it contains a copy of $J$.
	\end{theorem}
	We deduce:
	\begin{theorem}
		\label{thm:orderedniki}
		For every ordered graph $J$ and every $\vep\in(0,\frac12)$, there exists $\delta>0$ such that if $G$ is a ordered graph
		with $\ind_J(G)\le (\delta\abs G)^{\abs J}$,
		there is an $\vep$-restricted subset $S\subseteq V(G)$ with size at least $\delta \abs G$.
	\end{theorem}
	\begin{proof}
		Choose $H$ as in \cref{rodlwinkler}, and choose $\delta'$ such that setting $\delta=\delta'$ satisfies \cref{thm:niki}. 
		Let $h:=|H|$ and $j:=|J|$, and 
		let $\delta:=(\delta')^{h/j}$. We claim that $\delta$ satisfies the theorem. To show this, let 
		$G$ be an ordered graph
		with $\ind_J(G)\le (\delta\abs G)^{\abs J}$. We must show that
		there is an $\vep$-restricted subset $S\subseteq V(G)$ with size at least $\delta\abs G$.
		
		Since each copy of $J$ in $G$ extends to at most $|G|^{h-j}$ copies of $H$ in $G\nat$,
		and each copy of $H$ in $G\nat$ is an extension of some copy of $J$ in $G$ (because of the choice of $H$), there are at most 
		$$\ind_J(G)|G|^{h-j}\le (\delta\abs G)^j|G|^{h-j}= (\delta'\abs G)^h$$
		copies of $H$ in $G\nat$. But then the result follows from the choice of $\delta'$. This proves \cref{thm:orderedniki}.
	\end{proof}
	
	We observe:
	\begin{lemma}
		\label{lem:increasing}
		If $\mac F\subset\mac F'$ and $\mac F$ is viral, then so is $\mac F'$.
	\end{lemma}
	\begin{proof}
		This follows from the fact that $\mu_{\mac F}(x,G)\le\mu_{\mac F'}(x,G)$ for every $x>0$ and every graph (or ordered graph) $G$.
	\end{proof}
	
	It will also be useful to have an analogue of the Alon-Pach-Solymosi theorem (\cref{thm:aps}) for viral families of graphs.  This was proved for single graphs in~\cite{polyP4}; here is an extension to families: 
	\begin{lemma}
		\label{lem:sub}
		Let $\mac F$ be a finite set of graphs, and let $H_1,H_2$ be graphs such that $\mac F\cup\{H_1\}$ and $\mac F\cup\{H_2\}$ are viral.
		Let $H$ be obtained by substituting $H_2$ for a vertex $v$ of $H_1$.
		Then $\mac F\cup\{H\}$ is viral.
		The same is true for ordered graphs in place of graphs.
	\end{lemma}
	\begin{proof}
		The following proof works for both graphs and ordered graphs.
		If $H_1$ or $H_2$ is in $\mac F$ then $\mac F$ is viral, and so $\mac F\cup\{H\}$ is viral by \cref{lem:increasing}.
		Thus we may assume $H_1,H_2\nin\mac F$.
		For $i\in\{1,2\}$, let $h_i:=\abs{H_i}$, and let $d_i>0$ be a viral exponent for $\mac F_i:=\mac F\cup\{H_i\}$.
		We claim that $d:=d_1h_1+d_2+1$ is a viral exponent for $\mac F':=\mac F\cup\{H\}$.
		To see this, we may assume $h_1\ge2$.
		Now, let $\vep\in(0,1/2)$ and let $G$ be a graph (or ordered graph) with $\mu_{\mac F'}(\vep^d,G)\le 1$;
		and suppose for a contradiction that there is no $\vep$-restricted $S\subseteq V(G)$ with $\abs S\ge \vep^d\abs G$.
		Then $\mu_{\mac F_1}(\vep^{d_1},G)>1$ by the choice of $d_1$.
		Thus, since $\mu_{\mac F}(\vep^{d_1},G)\le \mu_{\mac F}(\vep^{d},G)\le \mu_{\mac F'}(\vep^{d},G)\le 1$,
		we deduce that $\mu_{H_1}(\vep^{d_1},G)>1$.
		It follows that
		\[\ind_{H_1}(G)>(\vep^{d_1}\abs G)^{h_1}
		=\vep^{d_1h_1}\abs G^{h_1}.\]
		For every copy $\varphi$ of $H_1\setminus v$ in $G$, let $I_{\varphi}$ be the set of copies $\varphi'$ of $H_1$ with $\varphi'\vert_{V(H_1\setminus v)}=\varphi$.
		Let $T$ be the set of copies $\varphi$ of $H_1\setminus v$ in $G$ with $\abs{I_{\varphi}}\ge\vep^{d_1h_1+1}\abs G$; then
		\[\sum_{\varphi\in T}\abs{I_{\varphi}}
		\ge \ind_{H_1}(G)-\abs G^{h_1-1}\cdot\vep^{d_1h_1+1}\abs G
		>\vep^{d_1h_1}\abs G^{h_1}
		-\vep^{d_1h_1+1}\abs G^{h_1}
		\ge \vep^{d_1h_1+1}\abs G^{h_1}.\]
		Thus $\abs T>\vep^{d_1h_1+1}\abs G^{h_1-1}\ge (\vep^d\abs G)^{h_1-1}$ since $h_1\ge2$ and $\abs{I_{\varphi}}\le\abs G$ for all $\varphi\in T$.
		\begin{claim}
			\label{claim:sub}
			For every $\varphi\in T$, there are at least $(\vep^d\abs G)^{h_2}$ copies $\varphi''$ of $H$ in $G$ with $\varphi''\vert_{V(H_1\setminus v)}=\varphi$.
		\end{claim}
		\begin{subproof}
			Let $A:=\{\varphi'(v):\varphi'\in I_{\varphi}\}$; then $\abs A\ge \vep^{d_1h_1+1}\abs G$.
			Thus, since $G$ includes no $\vep$-restricted $S\subseteq V(G)$ with $\abs S\ge\vep^d\abs G$,
			$G[A]$ contains no $\vep$-restricted $S\subseteq A$ with $\abs S\ge\vep^{d_2}\abs A\ge \vep^d\abs G$.
			The choice of $d_2$ then implies that
			$\mu_{\mac F_2}(\vep^{d_2},G[A])>1$.
			Hence, because
			\[\mu_{\mac F}(\vep^{d_2},G[A])
			\le\mu_{\mac F}(\vep^{d_1h_1+d_2+1},G)
			= \mu_{\mac F}(\vep^d,G)\le 1,\]
			we obtain $\mu_{H_2}(\vep^{d_2},G[A])>1$,		and so $\ind_{H_2}(G[A])>(\vep^{d_2}\abs A)^{h_2}\ge(\vep^d\abs G)^{h_2}$.
			Since each copy of $H_2$ in $G[A]$ together with $\varphi$ forms a copy $\varphi''$ of $H$ in $G$ with $\varphi''\vert_{V(H_1\setminus v)}=\varphi$ and these copies are distinct,
			the proof of \cref{claim:sub} is complete.
		\end{subproof}
		Now, \cref{claim:sub} yields
		\[\ind_H(G)\ge \abs T(\vep^d\abs G)^{h_2}
		>(\vep^d\abs G)^{h_1-1}(\vep^d\abs G)^{h_2}
		=(\vep^d\abs G)^{\abs H}\]
		and so $\mu_{\mac F'}(\vep^d,G)\ge \mu_H(\vep^d,G)>1$, a contradiction.
		This proves \cref{lem:sub}.
	\end{proof}
	
	\cref{lem:sub} allows us to focus our attention on prime ordered graphs.
	
	
	\section{Being divisive and being viral}
	
	For the next three sections, we will focus on ordered graphs and proving \cref{thm:mainorderedpair}. At the end of the paper, we look at its corollaries, for 
	unordered graphs, for tournaments, and for excluding one graph instead of two.  The goal of this section is to show that a finite set of graphs or ordered graphs is viral if and only if it has the property of being ``divisive''.  This will turn out to be easier to work with.
	
	Let $\mac F$ be a finite set of graphs or ordered graphs. We say that $\mac F$ is {\em weakly viral} if 
	there exists $d>0$ such that for every $\vep\in(0,\frac12)$,
	and for every graph (or ordered graph, appropiately) $G$ with $\mu_{\mac F}(\vep^d,G)\le 1$, there is a subset $S\subseteq V(G)$ with
	$\abs S\ge \vep^d\abs G$ such that  
	one of $G[S],\overline{G}[S]$ has at most $\vare\binom{|S|}{2}$ edges. (So we are not restricting the maximum degree
	in one of $G[S],\overline{G}[S]$, just the average degree.) 
	We call such a number $d$ a {\em weak viral exponent} for $\mac F$. It does not really 
	matter which definition we use, becuase of the following.
	
	\begin{lemma}\label{lem:viraleqnt}
		Let $\mac F$ be a finite set of graphs or ordered graphs.
		If $d$ is a viral exponent for $\mac F$ then it is a weak viral exponent for $\mac F$. Conversely, if $d$ is a 
		weak viral exponent for $\mac F$, then $3d$ is a viral exponent for $\mac F$. In particular, $\mac F$ is viral if and only if it is weakly viral.
	\end{lemma}
	\begin{proof}
		The first assertion is clear. For the second, let $d$ be a weak viral exponent for $\mac F$, let $\vep\in(0,\frac12)$, 
		and let $G$ be a graph or ordered graph with 
		$\mu_{\mac H}(\vep^{3d},G)\le 1$. Since $\vare/4\ge \vare^3$ it follows that 
		$\mu_{\mac H}((\vep/4)^{d},G)\le 1$. Since $d$ is a 
		weak viral exponent for $\mac F$, there is a subset $T\subseteq V(G)$ with
		$\abs T\ge (\vep/4)^d\abs G$ such that
		one of $G[T],\overline{G}[T]$ has at most $(\vare/4)\binom{|T|}{2}$ edges, and we may assume the first.
		Consequently, at most $|T|/2$ vertices in $T$ have degree in $G[T]$ more than $\vare|T|/2$;
		and so there exists $S\subseteq T$ with $|S|\ge |T|/2$ such that every vertex in $S$ has degree at most $\vare|T|/2\le \vare|S|$
		in $G[T]$ and hence in $G[S]$. Thus $S$ is $\vare$-restricted. This proves \cref{lem:viraleqnt}.
	\end{proof}

	We say that a finite set $\mac F$ of graphs is ``divisive'', if every graph $G$ either contains  many 
	copies of some member of $\mac F$, or admits a blockade that is both long and wide, and either sparse or dense. More exactly,
	$\mac F$ is {\em divisive}
	if there exist $b,c>0$ such that for every $x\in(0,c)$ and every graph $G$ with $\mu_{\mac F}(x^b,G)\le1$,
	there is an $x$-sparse or $(1-x)$-dense $(k,\floor{\abs G/k^b})$-blockade in $G$ where $k\in[2,1/x]$.
	(This property is a variant of the so-called {\em quasi-\erh{} property} employed in~\cite{tomon2020,MR4563865,MR4273057}.)
	Similarly, a finite set $\mac F$ of ordered graphs is {\em divisive}
	if there exist $b,c>0$ such that for every $x\in(0,c)$ and every ordered graph $G$ with $\mu_{\mac F}(x^b,G)\le1$,
	there is an $x$-sparse or $(1-x)$-dense $(k,\floor{\abs G/k^b})$-blockade in $G$ where $k\in[2,1/x]$.
	We call $(b,c)$ a pair of {\em divisive sidekicks} for $\mac F$.
	
	In this section we show that all finite divisive sets are viral; this is a key result, and will be crucial in the 
	proof of~\cref{thm:viral}. We will need the following:
	\begin{theorem}\label{transfer}
		Let $G$ be a graph, and let $\vare\in(0,\frac12)$ and $d\ge 1$. Let $x=\vare^{12d}$.
		Suppose that for every induced subgraph $F$ of $G$ with $|F|\ge \vare^{4d}|G|$, there is an
		$x$-sparse or $(1-x)$-dense blockade in $F$ of length $k\in[2,1/x]$ and width at least
		$\abs F/k^d$. Then
		there exists $S\subseteq V(G)$ with $\abs S\ge x^{d+1}\abs G$ such that one of $G[S],\overline{G}[S]$
		has at most $\vare\binom{|S|}{2}$ edges.
	\end{theorem}
	\begin{proof}
		We may assume that $|G|>x^{-d-1}$, since otherwise we may take $|S|\le 1$
		to satisfy the theorem.
		A {\em cograph} is a $P_4$-free graph.
		Let $J$ be a cograph on vertex set $\{1,2,\ldots,\abs J\}$, and for each $j\in V(J)$ let $A_j\subseteq V(G)$, pairwise disjoint. We call $\mathcal{L}=(J,(A_j:j\in V(J)))$
		a {\em layout}. A pair $\{u,v\}$ of distinct vertices of $G$ is {\em undecided} for a layout $(J,(A_j:j\in V(J)))$
		if there exists $j\in V(J)$ with $u,v\in A_j$; and {\em decided} otherwise. Thus, all pairs $\{u,v\}$ with $u\notin \bigcup_{j\in V(J)}A_j$ are decided.
		A decided pair $\{u,v\}$ is {\em wrong} for $(J,(A_j:j\in V(J)))$ if there are distinct
		$i,j\in V(J)$ such that $u\in A_i$, $v\in A_j$, and either
		\begin{itemize}
			\item $u,v$ are adjacent in $G$ and $i,j$ are nonadjacent in $J$; or
			\item $u,v$ are nonadjacent in $G$, and $i,j$ are adjacent in $J$.
		\end{itemize}
		We are interested in layouts in which the number of wrong pairs is only a small fraction of the number of decided pairs.
		Choose a layout $\mathcal{L}=(J,(A_j:j\in V(J)))$ satisfying the following:
		\begin{itemize}
			\item $|A_j|\ge \vare^{6d}|G|$ for each $j\in V(J)$;
			\item $\sum_{j\in V(J)}|A_j|^{1/d}\ge |G|^{1/d}$;
			\item the number of wrong pairs is at most $x$ times the number of decided pairs; and
			\item subject to these three conditions, $|J|$ is maximum.
		\end{itemize}
		(This is possible since we may take $\abs J=1$ and $A_1=G$ to satisfy the first three conditions.)
		\begin{claim}\label{claim:transfer1}
			\em We may assume that $|J|\le 4\vare^{-2}$.
		\end{claim}
		\begin{subproof}
			Suppose that $|J|\ge 4\vare^{-2}$. Since $J$ is a cograph, it has a clique or stable set $I$ of size at least $|J|^{1/2}\ge 2/\vare$,
			and by taking complements if necessary, we may assume that $I$ is a stable set. For each $i\in I$, choose
			$B_i\subseteq A_i$ with size $\lceil\vare^{6d}|G|\rceil$,
			and let $S=\bigcup_{i\in I}B_i$. Thus $|S|\ge (2\vare^{-1})\vare^{6d}|G|$.
			We claim that $G[S]$ has edge-density at most $\vare$. There are at most $|I|^{-1}\binom{|S|}{2}$ edges $uv$ of $G[S]$ such that
			$u,v\in B_i$
			for some $i\in I$; and the number of edges $uv$ of $G[S]$
			such that $u\in B_i$ and $v\in B_j$ for some distinct $i,j\in I$ is at most the number of wrong pairs of
			$\mathcal{L}$, and hence at most
			$$x\binom{G}{2}\le x|G|^2/2\le x(\vare^{1-6d}|S|/2)^2/2=x\vare^{2-12d}|S|^2/8\le x\vare^{2-12d}\binom{|S|}{2}/2.$$
			Hence the number of edges of $G[S]$ is at most
			$(|I|^{-1}+x\vare^{2-12d}/2)\binom{|S|}{2}\le \vare \binom{|S|}{2}$ since $|I|^{-1}\le \vare/2$ and $x\vare^{2-12d}/2\le \vare/2$.
			Moreover,
			$$|S|\ge \vare^{6d}|G|\ge x^{d+1}|G|,$$
			and so the theorem is satisfied.
			This proves \cref{claim:transfer1}.
		\end{subproof}

		We may assume that $|A_1|\ge |A_j|$ for all $j\in V(J)$. Since $\sum_{j\in V(J)}|A_j|^{1/d}\ge |G|^{1/d}$,
		and $|J|\le 4\vare^{-2}$ by \cref{claim:transfer1}, it follows that
		$|A_1|^{1/d}\ge (\vare^2/4)|G|^{1/d}$, that is,
		$$|A_1|\ge \vare^{2d}2^{-2d}|G|\ge \vare^{4d}|G|.$$

		By applying the hypothesis to
		$G[A_1]$,
		we deduce that
		there is an $x$-sparse or $(1-x)$-dense blockade $(B_1\LL B_k)$ in $G[A_1]$ where $k\in[2,1/x]$, with width at least
		$|A_1|/k^d$.
		By taking complements, we may assume that $(B_1\LL B_k)$ is $x$-sparse.
		\begin{claim}\label{claim:transfer2}
			$k\ge 2/\vare$.
		\end{claim}
		\begin{subproof}
			Suppose that $k\le 2/\vare\le \vare^{-2}$. Then each of the sets $B_1\LL B_k$ has size at least
			$|A_1|/k^d\ge \vare^{2d} |A_1|$.
			By substituting a $k$-vertex stable set for
			the vertex $1$ in $J$, and replacing $A_1$ by $B_1\LL B_k$, we obtain a new layout $\mathcal{L}'=(J', (A_j':j\in V(J')))$ say,
			where $|J'|>|J|$. We claim that this violates the choice of $\mathcal{L}$; and so we must verify that $\mathcal{L}'$ satisfies the
			first three bullets in the definition of $\mathcal{L}$.
			Each $B_j$ satisfies
			$$|B_j|\ge \vare^{2d} |A_1|\ge \vare^{6d} |G|,$$
			and so the first bullet is satisfied.
			For the second bullet, since $B_1\LL B_k$ all have size at least $|A_1|/k^d$, it follows that
			$$|B_1|^{1/d}+\cdots+|B_k|^{1/d}\ge |A_1|^{1/d},$$
			and so $\sum_{j\in V(J)}|A_j'|^{1/d}\ge |G|^{1/d}$.
			For the third bullet, let
			$P$ be the set of all decided pairs for $\mathcal{L}$, and $Q\subseteq P$ the set of wrong pairs for $\mathcal{L}$,
			and define $P',Q'$ similarly for $\mathcal{L}'$. Then $|Q|\le x|P|$. Let $R$ be the
			set of all pairs $\{u,v\}$ with $u,v\in A_1$ such that $u,v$ belong to different blocks
			of $(B_1\LL B_k)$. Then $R\subseteq  P'\setminus P$; and $Q'\setminus Q\subseteq R$; and $|Q'\setminus Q|\le x|R|$
			since $(B_1\LL B_k)$ is $x$-sparse. Hence
			$|Q'\setminus Q|\le x|P'\setminus P|$, and so
			$$|Q'|\le |Q|+|Q'\setminus Q|\le x|P|+ x|P'\setminus P|= x|P'|$$
			since $P\subseteq P'$.
			This contradicts the choice of $\mathcal{L}$, and so proves \cref{claim:transfer2}.
		\end{subproof}

		Let $n=\lceil 2/\vare\rceil$. For $1\le i\le n$, choose $C_i\subseteq B_i$ with size $w:=\ceil{|A_1|/k^d}$, uniformly at random.
		The probability that an edge between $B_i,B_j$ has ends in $C_i$ and $C_j$ is $\frac{w^2}{|B_i|\cdot|B_j|}$, and since there are
		at most $x|B_i|\cdot|B_j|$ edges between $B_i, B_j$, the expected number
		of edges between $C_i,C_j$ is at most $xw^2$. Markov's inequality then implies that the probability that there are more than $xn^2w^2/2$ such
		edges is less than $2/n^2$. It follows that the probability that for all distinct $i,j\in \{1\LL n\}$, there are at most
		$xn^2w^2/2$ edges between $C_i,C_j$ is positive, and so there is a choice of $C_1\LL C_n$ such that for all distinct
		$i,j$ there are at most
		$xn^2w^2/2$ edges between $C_i,C_j$. Let $S=C_1\cupcup C_n$. The number of edges of $G[S]$
		with ends in the same $C_i$ is at most $(1/n)\binom{|S|}{2}$; and the number of edges of $G[S]$ with ends in distinct blocks $C_i,C_j$
		is at most $(xn^2w^2/2)(n^2/2)=xn^2|S|^2/4\le xn^2\binom{|S|}{2} $. Consequently $G[S]$ has at most
		$(1/n+xn^2)\binom{|S|}{2}\le \vare \binom{|S|}{2}$ edges, since
		$1/n\le \vare/2$ and
		$xn^2\le x(4/\vare)^2\le \vare/2$.
		Moreover,
		$$|S|\ge w\ge |A_1|/k^d\ge \vare^{4d}|G|/k^d\ge \vare^{4d}x^d|G|\ge x^{d+1}|G|,$$
		and hence $S$ satisfies the theorem.
		This proves \cref{transfer}.
	\end{proof}
	
	
	This is used to prove 
	the following:
	\begin{theorem}
		\label{thm:divisive}
		If $\mac F$ is a finite set of graphs or ordered graphs, then it is divisive if and only if it is viral.
	\end{theorem}
	\begin{proof}
		We will only need the ``only if'' direction, but the ``if'' direction of this theorem is simple. To see this,
		we assume that  $\mac F$ is a viral set of graphs, or of ordered graphs.
		Let $d$ be a viral exponent for $\mac F$. Let $x>0$ with $x\le \min(1/16,1/d)$;
		and let $G$ be a graph or ordered graph with $\mu_{\mac F}(x^{2d},G)\le 1$. There exists
		an $x^2$-restricted $S\subseteq V(G)$ with $\abs S\ge x^{2d}\abs G$.
		Let $k:=\ceil{x^{-1/2}}\in[2,1/x]$, and
		choose a sequence $(B_1,\ldots,B_k)$ of disjoint subsets of $S$, all  of size $\floor{2x\abs S}$ (such subsets exist since
		$k\floor{2x\abs S}\le 4x^{1/2}\abs S\le\abs S $).
		Then $(B_1,\ldots,B_k)$ is an $x$-sparse or $(1-x)$-dense $(k,\floor{\abs G/k^{4d+2}})$-blockade in $G$ 
		(since $\floor{x^2\abs S}\le x\floor{2x\abs S}$). This proves the ``if'' direction.
		
		For the ``only if'' direction, we assume that $\mac F$ is a set of graphs or ordered graphs.
		Let $(b_0,c_0)$ be a pair of divisive sidekicks for $\mac F$, and let $d=\max(b_0,1/c_0)$ (so $(d,1/d)$ is also a pair 
		of divisive sidekicks for $\mac F$).
		Let $c:=12(d+1)(d+2)$. We claim that $c$ is a weak viral exponent for $\mac F$.
		
		To show this,
		let $\vare\in(0,\frac12)$
		and let $G$ be a graph (or ordered graph, if $\mac F$ is a set of ordered graphs) with
		$\mu_{\mac F}(\vare^c,G)\le1$. We must show that there exists $S\subseteq V(G)$ with $\abs S\ge \vare^c\abs G$ such that
		one of $G[S],\overline{G}[S]$ has at most $\vare\binom{|S|}{2}$ edges.
		We may assume that $|G|>\vare^{-c}$, since otherwise we may take $|S|\le 1$.
		Let $d'=d+1$, and $x=\vare^{12d'}$. We claim that:
		\begin{claim}
			\label{claim:divisive}
			For every induced subgraph $F$ of $G$ (or of $G\nat$, if $G$ is ordered) with $|F|\ge \vare^{4d'}|G|$, there is an
			$x$-sparse or $(1-x)$-dense blockade in $F$ of length $k\in[2,1/x]$ and width at least
			$\abs F/k^{d'}$.
		\end{claim}
		\begin{subproof}
			We observe that $x^d\vare^{4d'}=\vare^{12dd'+4d'}\ge \vare^c$, and so $x^d|F|\ge \vare^c|G|$.
			It follows that
			$$\mu_{\mac F}(x^d,F)\le \mu_{\mac F}(\vare^c,G)\le1.$$
			Since $d$ is a sidekick for the divisiveness of $\mac F$,
			there exists
			$k\in[2,1/x]$ such that there is an $x$-sparse or $(1-x)$-dense blockade in $F$ of length at least $k$ and width at least
			$\floor{\abs F/k^d}$.
			But
			$$|F|/k^d\ge x^d|F|\ge \vare^c|G|> 1,$$
			and so
			$$\floor{\abs F/k^d}\ge \abs F/(2k^d)\ge \abs F/k^{d+1}=|F|/k^{d'}.$$
			This proves \cref{claim:divisive}.
		\end{subproof}
		From \cref{transfer}, with $d$ replaced by $d'$, we deduce that
		there exists $S\subseteq V(G)$ with $\abs S\ge x^{d'+1}\abs G=\vare^c|G|$ such that one of $G[S],\overline{G}[S]$
		has at most $\vare\binom{|S|}{2}$ edges. 
		This proves that $\mac F$ is weakly viral, and hence viral by \cref{lem:viraleqnt}, and so proves \cref{thm:divisive}.
	\end{proof}

	\section{Using the leaf}

	In this section, we are given an ordered graph $H$ with a vertex $v$ of degree one that is an end 
	vertex of $H$, and a sparse ``host'' ordered graph $G$,
	and we would like to show that either:
	\begin{itemize}
		\item $G$ contains many copies of $H$; or
		
		\item we can locate a subset of $V(G)$ of decent size that induces an ordered subgraph containing not too many copies of $H\setminus \{v\}$; or
		
		\item there is a sparse blockade in $G$ with length and width polynomially related to the sparsity parameter.
	\end{itemize}
	To do this, we shall first prove this with the third outcome replaced by
	\begin{itemize}
		\item there are disjoint subsets  $A,B$ of $V(G)$ where $B$ is sparse to $A$, and $A$ has decent size, and $|B|$ contains almost all vertices of $G$.
	\end{itemize}
	Then, by iterating the procedure,  we will convert the two subsets of this last outcome into a sparse blockade of the desired length and width.
	
	We obtain sparse pairs by means of the following lemma:
	\begin{lemma}
		\label{lem:sparse1}
		Let $H$ be an ordered graph with an end vertex $v$ of degree one. Let $h:=\abs H\ge3$, and let $H':=H\setminus \{v\}$.
		Let $x,y>0$ with $x\le y\le\frac1{2h}$,
		and let $G$ be an ordered graph with maximum degree at most $y\abs G$.
		Then for every $a\ge2$, one of the following outcomes hold:
		\begin{itemize}
			\item $\ind_{H}(G)> x^{2a+h}\abs G^{h}$;
			
			\item there exists $S\subseteq V(G)$ with $|S|\ge y|G|$ such that $\ind_{H'}(G[S])\le y^{a-2}\abs S^{h-1}$; or
			
			\item there are disjoint $A,B\subseteq V(G)$ such that $\abs A\ge {y^a\abs G}$, $\abs B\ge(1-hy)\abs G$, and $B$ is $x$-sparse to $A$.
		\end{itemize}
	\end{lemma}	
	\begin{proof}
		We assume that the last two outcomes do not hold. If $y\abs G\le 1$ then the second outcome holds with $S$ being a singleton (since $\abs{H'}=h-1\ge 2$); and if $y^a\abs G\le 1\le y\abs G$ then the third outcome holds by letting $A=\{w\}$ for $w\in V(G)$ and $B$ be the set of nonneighbours of $w$ in $G$ (then $\abs B\ge \abs G-y\abs G-1\ge (1-hy)\abs G$).
		Thus $y^a\abs G\ge1$.
		From the symmetry, we may assume that $v$ is the last vertex of $H$.
		Let $u$ be the neighbour of $v$ in $H$, and let $J:=H\setminus \{u,v\}$.
		Let $S$ be the set of the first $\lceil y|G|\rceil$ vertices of $G$; that is, $S$ is the subset of $V(G)$
		with $|S|=\lceil y|G|\rceil$ such that $p\le_G q$ for all $p\in S$ and $q\in V(G)\setminus S$.
		For every copy $\varphi$ of $J$ in $G[S]$,
		let $I_{\varphi}$ be the set of copies $\varphi'$ of $H'$ in $G[S]$ with $\varphi'\vert_{V(J)}=\varphi$.
		Let $T$ be the set of copies $\varphi$ of $J$ in $G[S]$ with $\abs{I_{\varphi}}\ge y^a\abs G$.
		Since $y^a\abs G\le y^{a-1}\abs S$, since there are at most $\abs S^{h-2}$ copies of $J$ in $G[S]$, and since the second outcome of the lemma does not hold, we have (note that $y\in(0,\frac12)$)
		\[\sum_{\varphi\in T}\abs{I_{\varphi}}
		\ge \ind_{H'}(G[S])-\abs S^{h-2}\cdot y^a\abs G
		> y^{a-2}\abs S^{h-1}-y^{a-1}\abs S^{h-1}
		\ge y^{a-1}\abs S^{h-1},\]
		and so $\abs T> y^{a-1}\abs S^{h-2}$, since $\abs{I_{\varphi}}\le \abs S$ for all $\varphi\in T$.
		
		\begin{claim}
			\label{claim:sparse1}
			For every $\varphi\in T$ there are at least $x^{a+3}\abs G^2$ copies $\varphi''$ of $H$ in $G$ with $\varphi''\vert_{V(J)}=\varphi$.
		\end{claim}
		\begin{subproof}
			Let $P:=\varphi(V(J))$;
			then $\abs P= h-2$.
			Let $A':=\{\varphi'(v):\varphi'\in I_{\varphi}\}$;
			then $A'\subseteq S$ and $\abs{A'}=\abs{I_{\varphi}}\ge y^a\abs G$.
			Let $A\subseteq A'$ with $\abs A=\ceil{y^a\abs G}$.
			Now $\abs{V(G)\setminus S}= |G|-\lceil y|G|\rceil\ge |G|(1-y)-1$; and at most $y(h-2)|G|$ 
			vertices in $V(G)\setminus S$ have a neighbour in $P$.
			Let $B$ be the set of vertices in $V(G)\setminus S$ with no neighbours in $P$;~then 
			\[\abs B\ge (1-(h-1)y)\abs G-1.\]
			Since the third outcome does not hold, there are at most $(1-hy)\abs G$ vertices in $B$ that have fewer than $x\abs A$ neighbours in $A$.
			Thus the number of vertices in $B$ with at least $x\abs A$ neighbours in $A$ is at least
			\[\abs B-(1-hy)\abs G\ge y\abs G-1\ge 2y^2\abs G-1\ge y^2\abs G\]
			(note that $y\abs G\ge 2y^2\abs G\ge 2y^a\abs G\ge 2$).
			Consequently there are at least $xy^2\abs A\cdot\abs G\ge x^{a+3}\abs G^2$ edges between $A,B$.
			Since the endpoints of each such edge together with $P$ form a copy $\varphi'$ of $H$ in $G$ with $\varphi'\vert_{V(J)}=\varphi$,
			there are at least $x^{a+3}\abs G^2$ copies $\varphi'$ of $H$ in $G$ with $\varphi'\vert_{V(J)}=\varphi$, as claimed.
		\end{subproof}
		
		Therefore, \cref{claim:sparse1} implies that
		\[\ind_{H}(G)
		\ge \abs T\cdot x^{a+3}\abs G^2
		> y^{a-1}\abs S^{h-2} \cdot x^{a+3}\abs G^2
		\ge y^{a+h-3}x^{a+3}|G|^h
		\ge x^{2a+h}\abs G^{h}\]
		which is the first outcome.
		This proves \cref{lem:sparse1}.
	\end{proof}

	Now we turn sparse pairs into sparse, long, wide blockades.

	\begin{lemma}
		\label{lem:sparse2}
		Let $H$ be an ordered graph with $h:=\abs H\ge3$, and let $v\in V(H)$ be an end vertex of $H$, and have degree one.
		Let $H':=H\setminus \{v\}$.
		Let $x,y>0$ with $x\le y\le 4^{-h}$,
		and let $G$ be an ordered graph with maximum degree at most $y^2\abs G$.
		Then for every $a\ge2$, one of the following holds:
		\begin{itemize}
			\item $\ind_H(G)>x^{2a+2h}\abs G^{h}$;
			
			\item there exists $S\subset V(G)$ with $\abs S\ge y^2\abs G$ such that $\ind_{H'}(G[S])\le y^{a-2}\abs S^{h-1}$; or
			
			\item there is an $x$-sparse $(y^{-1},{y^{a+1}\abs G})$-blockade in $G$.
		\end{itemize}
	\end{lemma}
	\begin{proof}
		Suppose that none of the outcomes holds.
		If $y^2\abs G\le 1$ then the second outcome holds with $\abs S=1$ (since $\abs{H'}=h-1\ge2$);
		and if $y^{a+1}\abs G\le 1\le y^2\abs G$ then the third outcome holds since $G$ has a stable set of size at least $\frac{\abs G}{1+y^2\abs G}\ge \frac12y^{-2}\ge y^{-1}$.
		Hence $y^{a+1}\abs G\ge1$.
		Let $k\ge0$ be maximal such that there is an $x$-sparse blockade $(B_0,B_1,\ldots,B_{k})$ in $G$ 
		such that $\abs{B_{i-1}}\ge y^{a+1}\abs G$ for all $i\in[k]$, and $\abs{B_{ k}}\ge (1-hy)^{k}\abs G$. (For $k=0$ one can take $B_0:=V(G)$.)
		Since the third outcome does not hold, we have $ k<y^{-1}$, which implies
		(note that $1-t\ge 4^{-t}$ for all $t\in[0,\frac12]$ and $hy\le h\cdot 4^{-h}\le\frac12$) 
		\[\abs{B_{ k}}\ge(1-hy)^{ k}\abs G
		\ge 4^{-hy k}\abs G
		>4^{-h}\abs G\ge y\abs G\ge x\abs G.\]
		It follows that $G[B_{ k}]$ has maximum degree at most $y^2\abs G\le y\abs{B_{ k}}$.
		Since the first two outcomes do not hold, we have
		\begin{itemize}
			\item $\ind_{H}(G[B_{ k}])
			\le\ind_{H}(G)
			\le x^{2a+2h}\abs G^{h}
			\le x^{2a+h}\abs {B_{ k}}^{h}$; and
			
			\item $\ind_{H'}(G[B_{ k}])> y^{a-2}\abs {B_{ k}}^{h-1}$.
		\end{itemize}
		By \cref{lem:sparse1},
		there are disjoint $A,B\subseteq B_{ k}$ with
		\[\abs A\ge {y^a\abs {B_{ k}}}\ge {y^{a+1}\abs G}
		\quad
		\text{and}
		\quad
		\abs B\ge(1-hy)\abs{B_{ k}}
		\ge (1-hy)^{ k+1}\abs G\]
		such that $B$ is $x$-sparse to $A$.
		Therefore $(B_0,\ldots,B_{ k-1},A,B)$ is an $x$-sparse blockade, and this contradicts the maximality of $ k$.
		This proves \cref{lem:sparse2}.
	\end{proof}
	
	
	\section{Iterative sparsification}
	
	To finish the proof of \cref{thm:mainorderedpair}, we need to write out the argument sketched in \cref{sec:sketch}.
	\begin{theorem}\label{axioms}
		Let $\mac F$ be a finite set of ordered graphs, and let $H,\overline{J}\in \mathcal{F}$.
		Let $v$ be an end vertex of $H$ with degree one, and let $w$ be an end vertex of $J$ with degree one.
		Let $H'=H\setminus \{v\}$ and $J'=J\setminus \{w\}$. If $\{H'\}\cup (\mac F\setminus \{H\})$ and 
		$\{\overline{J'}\}\cup (\mac F\setminus \{\overline{J}\})$ are both viral then 
		$\mac F$ is viral.
	\end{theorem}
	\begin{proof}
		We may assume that $\abs H,\abs J\ge 3$. Define $h:=\max(\abs H,\abs J,4)$, and $c:=4^{-h}$. By \cref{thm:orderedniki}, there exists $c'>0$ such 
		that for every ordered graph $G$ with $\mu_{\mac F}(c',G)\le 1$, there is a $c^2$-restricted $S\subseteq V(G)$ with $\abs S\ge c'\abs G$.
		Since $\mac F_1:=\{H'\}\cup (\mac F\setminus \{H\})$ and $\mac F_2:=\{\overline{J'}\}\cup (\mac F\setminus \{\overline{J}\})$ are viral,
		there exists $d>0$ that is a viral exponent for them both; and
		we may increase $d$ so that $d\ge\max(4,\log_c(c'))$. 
		Let $a:=2d^2h$ and $b:=a+6d+1$.
		We need the following claim.
		\begin{claim}
			\label{claim:res}
			Let $x,y>0$ with $x\le y\le c$, and let $G$ be a $y^2$-restricted ordered graph.
			Then either
			\begin{itemize}
				\item $\mu_{\mac F}(x^{b-4d},G)>1$; or
				
				\item $\min(\mu_{H'}(y^{2d^2},G[S]),\mu_{\overline {J'}}(y^{2d^2},G[S]))\le 1$ for some $S\subseteq V(G)$ with $\abs S\ge y^2\abs G$; or
				
				\item there is an $x$-sparse or $(1-x)$-dense $(y^{-1},{y^{a+1}\abs G})$-blockade in $G$.
			\end{itemize}
		\end{claim}
		\begin{subproof}
			If $G$ has maximum degree at most $y^2\abs G$, then \cref{lem:sparse2} implies that either:
			\begin{itemize}
				\item $\ind_H(G)>x^{2a+2\abs H}\abs G^{\abs H}\ge (x^{b-4d}\abs G)^{\abs H}$, where the second 
				inequality is by the choice of $b$ (and so $\mu_{\mac F}(x^{b-4d},G)>1$); or
				
				\item $\ind_{H'}(G[S])\le y^{a-2}\abs S^{\abs{H'}}\le (y^{2d^2}\abs S)^{\abs {H'}}$ for some 
				$S\subseteq V(G)$ with $\abs S\ge y^2\abs G$, where the second inequality is because $a-2=2d^2h-2\ge 2d^2(h-1)$ by the choice of $a$ (and so $\mu_{H'}(y^{2d^2},G[S])\le 1$); or
				
				\item there is an $x$-sparse $(y^{-1},{y^{a+1}\abs G})$-blockade in $G$.
			\end{itemize}
			
			If $\overline G$ has maximum degree at most $y^2\abs G$, then similarly, also by \cref{lem:sparse2}, either:
			\begin{itemize}
				\item $\ind_{\overline J}(G)=\ind_J(\overline G)>x^{2a+2\abs J}\abs G^{\abs J}\ge (x^{b-4d}\abs G)^{\abs J}$; or
				
				\item $\ind_{\overline{J'}}(G[S])=\ind_{J'}(\overline G[S])\le y^{a-2}\abs S^{\abs {J'}}\le (y^{2d^2}\abs S)^{\abs{J'}}$ for some $S\subseteq V(G)$ with $\abs S\ge y^2\abs G$; or
				
				\item there is a $(1-x)$-dense $(y^{-1},{y^{a+1}\abs G})$-blockade in $G$.
			\end{itemize}
			
			This proves \cref{claim:res}.
		\end{subproof}
		
		We claim that
		$(b,c)$ is a pair of divisive sidekicks for $\mac F$; and hence $\mac F$ is divisive, and consequently viral by \cref{thm:divisive}.
		Thus, let $x\in(0,c)$, and let $G$ be a graph with $\mu_{\mac F}(x^b,G)\le1$.
		Suppose for a contradiction that there is no $x$-sparse or $(1-x)$-dense $(k,\floor{\abs G/k^b})$-blockade in $G$ with 
		$k\in[2,1/x]$; then $\abs G\ge x^{-b/2}$, because otherwise the blockade with $\lfloor 1/x\rfloor\ge x^{-1/2}$ empty blocks contradicts our supposition.
		Let $m\ge2$ be the least integer with $c^{d^{m-1}}\le x$;
		then $c^{d^{m-2}}\ge x$.
		\begin{claim}
			\label{claim:sequence}
			There is a nested sequence $V(G)=S_0\supseteq S_1\supseteq\cdots\supseteq S_m$
			such that
			$\abs{S_i}\ge c^{3d^{i}}\abs{S_{i-1}}$ and $S_i$ is $c^{2d^{i-1}}$-restricted in $G$
			for all $i\in[m]$.
		\end{claim}
		\begin{subproof}		
			Since $\mu_{\mac F}(c',G)\le\mu_{\mac F}(c^d,G)\le\mu_{\mac F}(x^b,G)\le1$ by the choice of $b,d,x$,
			there exists a $c^2$-restricted $S_1\subseteq V(G)$ with $\abs{S_1}\ge c'\abs G\ge c^d\abs G$.
			This proves the base case.
			
			Now, for $i\in[m]$ with $i<m$,
			assuming that $S_0,S_1,\ldots,S_i$ have been constructed, we shall construct $S_{i+1}$.
			Let $y:=c^{d^{i-1}}\ge c^{d^{m-2}}\ge x$.
			Since $d^{i-1}+d^{i-2}+\cdots+1=\frac{d^i-1}{d-1}<\frac13d^i$ (as $d\ge4$), and consequently
			$3d^{i}+3d^{i-1}+\cdots+3d+3<4d^i$, we have
			\[\begin{aligned}
				\abs {S_i}\ge c^{3d^i}\abs{S_{i-1}}
				\ge c^{3d^i+3d^{i-1}}\abs{S_{i-2}}
				\ge\cdots
				&\ge c^{3d^i+3d^{i-1}+\cdots+3d}\abs {S_0}
				\ge c^{4d^{i}}\abs G 
				= y^{4d}\abs G
				\ge x^{4d}\abs G.
			\end{aligned}\]
			It follows that
			\[\mu_{\mac F}(x^{b-4d},G[S_i])
			\le \mu_{\mac F}(x^b,G)\le 1.\]
			Thus, since $S_i$ is $y^2$-restricted in $G$ and $0<x\le y\le c$,
			\cref{claim:res} implies that either
			\begin{itemize}
				\item $\min(\mu_{H'}(y^{2d^2},G[S]),\mu_{\overline {J'}}(y^{2d^2},G[S]))\le 1$ for some $S\subseteq S_i$ with $\abs S\ge y^2\abs {S_i}$; or
				
				\item $G[S_i]$ contains an $x$-sparse or $(1-x)$-dense $(y^{-1},{y^{a+1}\abs {S_i}})$-blockade.
			\end{itemize}
			If the second bullet holds, then since $y^{a+1}\abs{S_i}\ge y^{a+6d+1}\abs G\ge y^b\abs G$ by the choice of $b$,
			$G[S_i]$ (and thus $G$) contains an $x$-sparse or $(1-x)$-dense $(k,\floor{\abs G/k^b})$-blockade where $k=1/y\in[2,1/x]$,
			contradicting our supposition.
			Thus the first bullet holds.
			Let $G':=G[S]$.
			Because $a=2d^2h$ and $h=4$, we have $b=a+6d+1\ge 2d^2+5d$.
			Thus, since $\abs S\ge y^2\abs{S_i}\ge y^{4d+2}\abs G\ge y^{5d}\abs G$, we obtain
			\[\mu_H(y^{2d^2},G')
			\le \mu_H(y^{2d^2+5d},G)
			\le \mu_H(x^{b},G)
			\le 1,\]
			and similarly
			$\mu_{\overline J}(y^{2d^2},G')\le 1$.
			Therefore, recalling that $\mac F_1=\{H',\overline J\}$ and $\mac F_2=\{H,\overline{J'}\}$, 
			we deduce that $\mu_{\mac F_i}(y^{2d^2},G')\le 1$ for some $i\in\{1,2\}$.
			Hence, the choice of $d$ implies that
			$G'$ (and so $G$) contains a $y^{2d}$-restricted $S_{i+1}\subseteq S$ with 
			$\abs {S_{i+1}}\ge y^{2d^2}\abs S\ge y^{2d^2+2}\abs{S_i}
			\ge c^{3d^{i+1}}\abs{S_i}$.
			Since $y^{2d}=c^{2d^i}$, this completes the induction step and proves \cref{claim:sequence}.
		\end{subproof}
		Now, we have $x^{2d}\le c^{2d^{m-1}}\le x^2$, which implies that $S_m$ is $x^2$-restricted in $G$. Furthermore, because $b=a+6d+1=2d^2h+6d+1\ge 8d^2+1$, we see that
		\[\begin{aligned}
			\abs{S_m}\ge c^{3d^m}\abs{S_{m-1}}
			\ge\ldots &\ge c^{3d^m+3d^{m-1}+\cdots+3d}\abs{S_0}\\
			&\ge c^{4d^m}\abs{G}
			\ge x^{4d^2}\abs G
			\ge x^{4d^2-b/2}\ge x^{-1}.
		\end{aligned}\]
		Let $k:=\ceil{x^{-1/2}}$;
		then $k\ge x^{-1/2}\ge2$.
		Because $k\floor{2x\abs{S_m}}\le 2x^{-1/2}\cdot 2x\abs{S_m}=4x^{1/2}\abs{S_m}
		\le\abs{S_m}$,
		there are disjoint subsets $B_1,\ldots,B_k$ of $S_m$ with $\abs{B_i}=\floor{2x\abs{S_m}}$ for all $i\in[k]$.
		Since \[\floor{2x\abs{S_m}}\ge x\abs{S_m}\ge x^{4d^2+1}\abs G\ge \abs G/k^{8d^2+2}\ge \abs G/k^b\]
		by the choice of $b$,
		and since $S_m$ is $x^2$-restricted in $G$,
		$(B_1,\ldots,B_k)$ is an $x$-sparse or $(1-x)$-dense $(k,\abs G/k^b)$-blockade in $G$, a contradiction.
		
		This proves our claim that $(b,c)$ is a pair of divisive sidekicks for $\mac F$.
		Consequently $\mac F$ is divisive, and therefore viral by \cref{thm:divisive}.
		This proves \cref{axioms}. 
	\end{proof}
	
	Finally, we have:
	
	\begin{proof}
		[Proof of \cref{thm:mainorderedpair}]
		We proceed by induction on $\abs H+\abs J$.
		If $\min(\abs H,\abs J)\le 2$ then we are done by \cref{lem:increasing}, since every ordered graph on at most two vertices is viral.
		Let $\abs H,\abs J\ge 3$; we assume that the theorem is true for every pair $\{H',\overline{J'}\}$ with $\abs{H'}+\abs{J'}<\abs H+\abs J$, and
		we shall prove it for $\mac F:=\{H,\overline J\}$.
		
		If one of $H,J$ is not prime, then $\{H,\overline J\}$ is viral by \cref{lem:sub} and the induction hypothesis.
		Thus we may assume they are both prime; and so there is a vertex $v$ of $H$ with degree one, and $v$ is either the first or 
		last vertex of $H$. Choose $u\in V(J)$ similarly, and let
		$H':=H\setminus \{v\}$ and $J':=J\setminus \{u\}$.
		From the inductive hypothesis, 
		$\mac F_1:=\{H',\overline J\}$ and $\mac F_2:=\{H,\overline{J'}\}$ are~viral; and so \cref{axioms} implies that
		$\{H,\overline{J}\}$ is viral. This proves 
		\cref{thm:mainorderedpair}.
	\end{proof}

	
	\section{Corollaries}\label{sec:cor}

	There are several corollaries of these results. First, let us say a {\em monotone path} is an ordered graph
	$H$ such that $H\nat$ is a path with vertex set  $\{v_1,\ldots, v_k\}$, where $v_i,v_{i+1}$ are adjacent for $1\le i<k$, and $v_i\le_H v_j$ if $i<j$.
	As we mentioned earlier, Pach and Tomon~\cite{MR4273057} proved:
	\begin{theorem}\label{pachtomon}
		If $H_1,H_2$ are both monotone paths, then there exists $c>0$ such that if $G$ is an ordered graph that is both $H_1$-free and $\overline{H_2}$-free then $G$ has a clique or stable set of size at least $|G|^c$.
	\end{theorem}
	Since monotone paths belong to $\mathcal{K}$, this theorem is a special case of \cref{thm:mainorderedpair}. Indeed, it follows that if $H_1,H_2$ are monotone paths then $\{H_1,\overline{H_2}\}$ is viral.
	
	Second, what happens if we insist that $H_1=\overline{H_2}$ in \cref{thm:mainorderedpair}? That gives us a class of
	ordered graphs with the \erh{} property, as follows.
	Let $\mathcal{L}$ be the class of ordered graphs $G$ minimal such that:
	\begin{itemize}
		\item the one-vertex ordered graph is in $\mathcal{L}$;
		\item if $H_1,H_2\in \mathcal{L}$ and $H$ is obtained from $H_2$ by substituting $H_1$ for one of its vertices, then $H\in \mathcal{L}$;
		\item if $H\in \mathcal{L}$,
		and the first vertex $a$ of $H$ is adjacent to only the last vertex of $H$, then we can add a new last vertex adjacent to 
		all vertices in $V(H)\setminus \{a\}$, and this enlarged ordered graph is also in $\mathcal{L}$;
		\item also three variants of the bullet above, exchanging ``first'' with ``last'', and/or exchanging ``adjacent'' with ``nonadjacent'', which we do not write out explicitly.
	\end{itemize}
	
	If $H\in \mac L$, then evidently $H,\overline H\in\mac K$ (in fact $H\in\mac L$ if and only if $H,\overline H\in\mac K$; 
	this is proved like \cref{lem:char} below). We deduce from \cref{thm:mainorderedpair} that:
	\begin{theorem}\label{thm:orderedsingle}
		All ordered graphs in $\mathcal{L}$ are viral.
	\end{theorem}
	
	Let us say that an ordered graph $H$ has the {\em \erh{} property} if there exists $c>0$ such that if $G$ is an $H$-free
	ordered graph, then there is a clique or stable set in $G$ with size at least $|G|^c$. And as with unordered graphs,
	if an ordered graph $H$ is obtained by vertex-substitution from two smaller ordered graphs with the \erh{} property, then $H$ has the \erh{} property. So we would like to find prime ordered graphs that have the \erh{} property.
	It is striking that, until now, there were none known with more than three vertices. In \cref{fig:fourprime} we show all the 
	prime ordered graphs on four vertices (up to taking complements and reversing the order); as we said, none of them were previously known to 
	have the \erh{} property. Our result shows that the first, the fifth and the seventh have the property, and indeed are viral,
	because they belong to $\mathcal{L}$.
	\begin{figure}[h]
		\centering
		
		\tikzstyle{every node}=[inner sep=1.5pt, fill=black,circle,draw]
		\begin{tikzpicture}[scale=0.8,auto=left]
			\node (a1) at (0,7) {};
			\node (a2) at (2,7) {};
			\node (a3) at (4,7) {};
			\node (a4) at (6,7) {};
			\node (b1) at (0,6) {};
			\node (b2) at (2,6) {};
			\node (b3) at (4,6) {};
			\node (b4) at (6,6) {};
			\node (c1) at (0,5) {};
			\node (c2) at (2,5) {};
			\node (c3) at (4,5) {};
			\node (c4) at (6,5) {};
			\node (d1) at (0,4) {};
			\node (d2) at (2,4) {};
			\node (d3) at (4,4) {};
			\node (d4) at (6,4) {};
			\node (e1) at (0,3) {};
			\node (e2) at (2,3) {};
			\node (e3) at (4,3) {};
			\node (e4) at (6,3) {};
			\node (f1) at (0,2) {};
			\node (f2) at (2,2) {};
			\node (f3) at (4,2) {};
			\node (f4) at (6,2) {};
			\node (g1) at (0,1) {};
			\node (g2) at (2,1) {};
			\node (g3) at (4,1) {};
			\node (g4) at (6,1) {};
			
			\foreach \from/\to in {a1/a2,b1/b2,d1/d2,d2/d3,d3/d4,e1/e2,e3/e4,f1/f2,f3/f4,g1/g2,g2/g3}
			\draw [-] (\from) -- (\to);
			\foreach \from/\to in {b2/b4, c1/c3,c2/c4,e1/e3}
			\draw [-] (\from) [bend left=15] to (\to);
			\foreach \from/\to in {a1/a4,g1/g4,f1/f4}
			\draw [-] (\from) [bend left=20] to (\to);

		\end{tikzpicture}
		\caption{
			The four-vertex prime ordered graphs (up to complements and reversal).
		} \label{fig:fourprime}
	\end{figure}
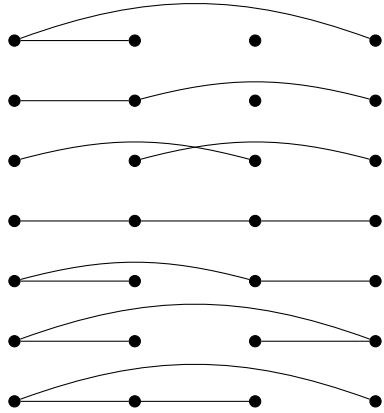
	
	Next, let us look at corollaries of \cref{thm:mainorderedpair} for unordered graphs.
	We will use a nice correspondence between the classes $\mac L, \mac K$ of ordered graphs and the classes $\mac H, \mac J$
	of (unordered) graphs:
	\begin{lemma}\label{ordering}
		If $G\in \mathcal{K}$ then $G\nat\in \mathcal{J}$; and if $F\in \mathcal{J}$, there is a linear order $\le $ of $V(F)$
		such that $(F,\le)\in \mathcal{K}$. Similarly, if $G\in \mathcal{L}$ then $G\nat\in \mathcal{H}$; and if $F\in \mathcal{H}$, there is a linear order $\le $ of $V(F)$
		such that $(F,\le)\in \mathcal{L}$.
	\end{lemma}
	The proofs are straightforward arguments by induction and we omit them.
	We will also need:
	\begin{lemma}\label{loseorder}
		Let $\mathcal{F}$ be a finite set of ordered graphs, and let $\mathcal{F}\nat$ be the set $\{F\nat:F\in \mathcal{F}\}$. If $\mathcal{F}$
		is viral then $\mathcal{F}\nat$ is viral.
	\end{lemma}
	\begin{proof}
		Let $d$ be a viral exponent for $\mathcal{F}$; we claim $d$ is also a viral exponent for $\mathcal{F}\nat$.
		Let $\vep\in(0,\frac12)$, and let $G$ be a graph with $\mu_{\mac F\nat}(\vep^d,G)\le 1$.
		Choose a linear order of $V(G)$, to obtain an ordered graph $H$ with $H\nat=G$.
		Thus $\mu_{\mac F}(\vep^d,G)\le \mu_{\mac F\nat}(\vep^d,G)\le 1$, and so
		there is an $\vep$-restricted $S\subseteq V(G)$ with $\abs S\ge \vep^d\abs G$. This proves \cref{loseorder}.
	\end{proof}
	
	From \cref{ordering} and \cref{loseorder}, we see that \cref{thm:mainorderedpair} implies:
	
	\begin{theorem}\label{thm:mainunordered}
		If $H_1,H_2\in \mathcal{J}$, then $\{H_1,\overline{H_2}\}$ is viral.
	\end{theorem}
	This implies \cref{thm:pairsEH}; and by
	taking $H_1=H_2$, we deduce from \cref{lem:char} below (mentioned following \cref{thm:pairsEH}) that \cref{thm:mainunordered} implies
	\cref{thm:viral}.

	\begin{lemma}
		\label{lem:char}
		$H\in\mac H$ if and only if $H,\overline H\in\mac J$.
	\end{lemma}
	\begin{proof}
		We first prove the ``only if'' direction by induction.
		Let $H\in\mac H$.
		If $\abs H\le 3$ then obviously $H,\overline H\in \mac J$,
		so we may assume $\abs H\ge 4$,
		and every induced subgraph $J$ of $H$ with $\abs J<\abs H$ satisfies $J,\overline J\in \mac J$.
		If $H$ is not prime then $H,\overline H\in\mac J$ by definition,
		so we assume it is prime. From the definition of $\mac H$, we may assume (replacing $H$ by $\overline{H}$ if necessary)
		that there is a vertex $v$ of $H$ with degree one, such that its neighbour $u$ is adjacent to all except one vertex of $H\setminus \{v\}$.
		So $H\in \mac J$; and also $\overline{H}\in \mac J$, since $u$ has degree one in $\overline H$.
		This completes the inductive step, and so prove the ``only if'' part.
		
		Now, we prove the ``if'' direction, again by induction.
		Let $H$ be such that $H,\overline H\in\mac J$.
		If $\abs H\le 3$ then $H\in\mac H$ and we are done;
		so we may assume $\abs H\ge 4$.
		Thus, since $H\in\mac J$,
		either $H$ contains a vertex of degree one or $H$ is not prime.
		If $H$ is not prime, then it can be obtained by substituting some graph $H_2$ for a vertex of some graph $H_1$,
		where $H_1,H_2$ are induced subgraphs of $H$ with $\abs{H_1},\abs{H_2}<\abs{H}$.
		Hence $H_i,\overline{H_i}\in\mac J$ for all $i\in\{1,2\}$; and so the induction hypothesis implies that $H_1,H_2\in\mac H$, which yields $H\in\mac H$.
		Hence, we may assume that $H$ has a vertex $u$ of degree one; and similarly, we may assume that $\overline H$ has a vertex $v$ of degree one.
		Since $\abs H\ge 4$, $u\ne v$.
		If $uv\in E(H)$ then $v$ is the unique neighbour of $u$ in $H$, and so $H\in\mac H$ since $H\setminus u\in\mac H$ by induction.
		If $uv\nin E(H)$ then $u$ is the unique neighbour of $v$ in $\overline H$, and so $\overline H\in\mac H$ since $\overline H\setminus v=\overline{H\setminus v}\in \mac H$ by induction.
		Therefore $H\in\mac H$, and this proves \cref{lem:char}.
	\end{proof}
	Now we prove that every prime graph in $\mac H$ is a split graph, which was the last bullet prior to \cref{thm:eh}. 
	\begin{lemma}
		\label{lem:prime}
		Every prime graph $H\in\mac H$ is split.
	\end{lemma}
	\begin{proof}
		The proof is by induction on $\abs H$. For $z\in V(H)$ and $Z\subset V(H)\setminus\{z\}$, we say that $z$ is {\em complete to} $Z$ in $H$ if it is adjacent in $H$ to every vertex in $Z$, and {\em pure to} $Z$ if it has no neighbour in $Z$ or no nonneighbour in $Z$.
		For $S\subset V(H)\setminus Z$, we say that $S$ is {\em complete to} $Z$ in $H$ if $z$ is complete to $Z$ in $H$ for all $z\in S$, and that $Z$ is {\em homogeneous} in $H$ if every vertex in $V(H)\setminus Z$ is pure to $Z$.
		
		Recall that $H$ has a vertex $u$ of degree $\abs H-2$ and a vertex $v$ of degree one in $H$; and we may assume $u,v$ are adjacent in $H$ and $\abs H\ge 5$. Let $w$ be the unique nonneighbour of $u$ in $H$.	
		Let $X\subset V(H)\setminus\{u,v\}$ be maximal such that $X$ is a clique and is complete to $V(H)\setminus
		(X\cup \{v\})$ in $H$. Then $X$ is homogeneous in $H$; and so $\abs X\le 1$ since $H$ is prime. Let $J:=H\setminus X$; then $\abs J\ge 4$ and $u,w\in V(J)$. Thus every homogeneous set of $J$ not containing $v$ is also homogeneous in $H$.
		
		\begin{claim}
			\label{claim:prime1}
			We may assume that $J\setminus v$ is not prime.
		\end{claim}
		\begin{subproof}
			Suppose that $J\setminus v$ is prime.
			Then by induction, there is a partition $V(J)\setminus\{v\}=S\cup K$ where $S$ is a stable set in $J\setminus v$ and $K$ is a clique in $J\setminus v$.
			Since $\abs{J\setminus v}\ge \abs H-2\ge 3$, every vertex in $S$ has a neighbour in $K$.
			If $u\in S$, then either $w\in S$ which implies $S=\{u,w\}$ and $u$ is complete to $K$, or $w\in K$ which yields $S=\{u\}$ and $\abs K\ge 2$. In either case there would be some $x\in K$ complete to $S$ and hence to $V(J)\setminus\{v,x\}$, contrary to the primeness of $J\setminus v$ since $\abs{J\setminus\{v,x\}}\ge 2$. Thus $u\in K$; and so $V(H)=(S\cup\{v\})\cup (K\cup X)$ is a desired partition of $V(H)$.
			This proves \cref{claim:prime1}.
		\end{subproof}
		
		Now, by \cref{claim:prime1}, there is a homogeneous set $Z$ in $J\setminus v$ with $1<\abs Z<\abs J$; we may assume that $\abs Z$ is maximal. If $u\nin Z$ then $Z$ would be homogeneous in $J$ (and so in $H$), contrary to the primeness of $H$; hence $u\in Z$. We claim that:
		\begin{claim}
			\label{claim:prime4}
			$w\nin Z$.
		\end{claim}
		\begin{subproof}
			Suppose that $w\in Z$. Then $u$ is complete to $V(J\setminus v)\setminus Z$; and so $V(J\setminus v)\setminus Z$ is complete to $Z$, since $Z$ contains $u$ and is homogeneous in $J\setminus v$. It follows that $V(J\setminus v)\setminus Z$ is a homogeneous set in $J$ and so in $H$. Then $\abs{(J\setminus v)\setminus Z}=1$ since $H$ is prime; but then $X\cup (V(J\setminus v)\setminus Z)$ contradicts the maximality of $X$. This proves \cref{claim:prime4}.
		\end{subproof}
		
		Now, let $F:=J\setminus\{u,v\}$; then $\abs F\ge 2$. We prove that $F$ is prime as follows.
		\begin{claim}
			\label{claim:prime2}
			$F$ is prime; in particular $\abs F\ge 4$.
		\end{claim}
		\begin{subproof}
			Suppose not; then $F$ has a homogeneous set $L$ with $1<\abs L<\abs F=\abs J-2$. If $w\nin L$ then $u$ is complete to $L$ in $H$; and so $L$ would be homogeneous in $H$, a contradiction. Hence $w\in L$.

			By \cref{claim:prime4}, $u$ is complete to $Z\setminus \{u\}$ in $H$; and so $\abs Z=2$ since $H$ is prime.
			Let $Z=\{u,y\}$.
			If $y\nin L$, then $L$ contains a neighbour of $u$ since $\abs L>1$, which implies that $y$ is complete to $L$ in $H$; and so $w$ has a neighbour (which is $y$) and a nonneighbour (which is $u$) in $Z$, a contradiction.
			Thus $y\in L\cap Z$; and so 
			every vertex in $V(J\setminus v)\setminus(L\cup Z)$ (if exists) is pure to $L\cup Z$ via $y$, which implies that $L\cup Z$ is homogeneous in $J\setminus v$. Since $w\in L\setminus Z$, we have $\abs{L\cup Z}>\abs Z$; and so $V(J\setminus v)=L\cup Z$ by the maximality of $Z$. But then
			$\abs J=\abs{L\cup Z\cup\{v\}}\le\abs L+\abs Z<\abs J$, a contradiction.
			This proves \cref{claim:prime2}.
		\end{subproof}
		By \cref{claim:prime2} and induction, there is a partition $V(F)=S\cup K$ where $S$ is a stable set in $F$ and $K$ is a clique in $F$. Because $F$ is prime and $\abs F\ge 4$, $S$ is nonempty and every vertex in $S$ has a neighbour in $K$; and so $\abs K\ge2$. We claim that:
		\begin{claim}
			\label{claim:prime3}
			$w\in S$.
		\end{claim}
		\begin{subproof}
			Suppose that $w\in K$. Since $w\nin Z$ by \cref{claim:prime4} and $u\in Z$, $w$ is anticomplete to $Z$. Thus $K,Z$ are disjoint; and so $Z$ intersects $S$ because $\abs Z\ge 2$.
			If there exists $z'\in S\setminus Z$ then $z'$ is adjacent to $u$ and anticomplete to $S\cap Z$, a contradiction since $Z$ is homogeneous in $J\setminus v$.
			Hence $S\subset Z$; and so $Z=S\cup\{u\}=V(J\setminus v)\setminus K$ is complete to $K\setminus\{w\}$. But then $X\cup (K\setminus\{w\})$ violates the maximality of $X$ since $\abs K\ge 2$.
			%
			%
			This proves \cref{claim:prime3}.
		\end{subproof}
		Now, \cref{claim:prime3} implies that $V(H)=(S\cup\{v\})\cup(K\cup X\cup\{u\})$ is a desired partition of $V(H)$.
		This proves \cref{lem:prime}.
	\end{proof}
	In view of \cref{lem:prime} and~\cite[Theorem 3.3]{MR3416129}, it would be interesting to know if every split graph has the \erh{} property, which would unify \cref{thm:eh} and~\cite[Theorem 1.2]{density6}. 
	
	There is also a version of \cref{axioms} for unordered graphs:
	\begin{theorem}\label{unorderedaxioms}
		Let $\mac F$ be a finite set of graphs, and let $H_1,\overline{H_2}\in \mathcal{F}$.
		For $ i = 1,2$, let $v_i$ be a vertex of $H_i$ with degree one, 
		and let $H_i'=H_i\setminus \{v_i\}$. If $\mac F_1:=\{H_1'\}\cup (\mac F\setminus \{H_1\})$ 
		and $\mac F_2:=\{\overline{H_2'}\}\cup (\mac F\setminus \{\overline{H_2}\})$ are both viral then
		$\mac F$ is viral.
	\end{theorem}
	\begin{proof}
		This can be proved directly, in the same way that we proved \cref{axioms}, but it can also be derived from \cref{axioms}, 
		as follows.
		Let $\mac P$ be the set of all ordered graphs $J$ such that $J\nat\in \mac F$, and define
		$\mac P_1$ and $\mac P_2$ similarly. Thus $\mac P_1$ and $\mac P_2$ are viral. 
		Let $\mac Q_1$ be the set of all $J$ in $\mac P$ such that $J\nat=H_1$ and $v_1$ is an end vertex of $J$; and let $\mac Q_1'$
		be the set of all ordered graphs $J\setminus \{v_1\}$ where $J\in \mac Q_1$. 
		Let $\mac Q_2$ be the set of all $J\in \mac P$ such that  $J\nat=\overline{H_2}$ and $v_2$ is an end vertex of $J$; 
		and let $\mac Q_2'$
		be the set of all ordered graphs $J\setminus \{v_2\}$ where $J\in \mac Q_2$. Thus 
		$\mac P_i\subseteq (\mac P\setminus Q_i)\cup Q_i'$ for $i = 1,2$.
		Since $\mac P_i$ is viral, it follows that $(\mac P\setminus Q_i)\cup Q_i'$ is viral, for $i = 1,2$. By repeated application of \cref{axioms},
		it follows that $\mac P$ is viral; and hence $\mac F$ is viral, by \cref{loseorder}.
		This proves \cref{unorderedaxioms}.
	\end{proof}

	Since $C_5$ is viral (see~\cite{bfp2024}), we can use \cref{unorderedaxioms} to obtain viral pairs of prime 
	graphs that are not given by \cref{thm:mainunordered}. For instance, say a graph $H$ is 
	{\em five-unicyclic} if all its cycles have length five, and every component has at most one cycle.
	Then \cref{unorderedaxioms}
	implies that if $H_1,H_2$ are both five-unicyclic, then $\{H_1,\overline{H_2}\}$ is viral. Similarly, if we assume that $P_5$ 
	is viral, we can use \cref{unorderedaxioms} to prove the viral property
	of
	the prime six-vertex graph obtained from $P_5$ by adding a new vertex adjacent to all the                                    
	old vertices except one end of the path.

	There are also implications for tournaments. A {\em tournament} is a digraph such that for all distinct $uv$, exactly one of 
	$uv,vu$ is an edge. 
	For tournaments $T,Q$, a {\em copy} of $Q$ in $T$ is an isomorphism from $Q$ to an subtournament of $T$, and $T$ is {\em $Q$-free} if there is no copy of $Q$ in $T$.
	Let $\ind_Q(T)$ denote the number of copies of $Q$ in $T$.
	We say that $Q$ has the {\em \erh{} property} if there exists $c>0$ such that every tournament $T$ admitting no copy of 
	$Q$ contains a transitive subtournament on at least $|T|^c$ vertices; 
	and Alon, Pach, and Solymosi~\cite{MR1832443} proved that the Erd\H{o}s-Hajnal conjecture is equivalent to the statement that every tournament has the \erh{} property.
	Let $\mac Q$ be the family of tournaments defined as follows:
	\begin{itemize}
		\item $\mac Q$ is closed under taking vertex-substitutions; and
		
		\item if $Q$ is a tournament, and $v\in V(Q)$ has indegree at most one or out-degree at most one, and 
		$Q\setminus \{v\}\in \mac Q$, then $Q\in \mac Q$.
	\end{itemize}
	Let $Q$ be a tournament, and take a numbering of its vertex set, say $\{v_1\LL v_n\}$.
	Let $G$ be the graph with vertex set $V(Q)$ in which for $1\le i<j\le n$, $v_i,v_j$ are adjacent in $G$ if and only if $v_i$ 
	is adjacent from $v_j$ in $Q$. We call $G$ the {\em backedge graph} of $Q$, and if we order its vertex set by $v_i\le v_j$ 
	if $i\le j$,
	we obtain an ordered graph called the {\em backedge ordered graph}.
	It is easy to prove by induction (again, we omit the details) that:
	\begin{lemma}\label{lem:QtoL}
		$Q\in \mathcal{Q}$ if and only if there is a numbering of its vertex set such that the resulting backedge ordered graph is in $\mathcal{K}$.
	\end{lemma}
	Then we can prove the tournament analogue of \cref{thm:eh}, which says:
	\begin{theorem}\label{tourEH}
		Every tournament in $\mac Q$ has the \erh{} property.
	\end{theorem}
	\begin{proof} Let $Q\in \mac Q$, and take a numbering of $V(Q)$ as in \cref{lem:QtoL}. Let $H\in \mathcal{K}$ be the 
		resulting backedge ordered graph. Let $H'$ be the ordered graph obtained from $H$ by reversing the linear order. 
		It follows that $H'\in \mathcal{K}$, since $\mathcal{K}$ is closed under taking reversals. Moreover, $\overline{H'}$ is also 
		a backedge ordered graph of $Q$, obtained from the reversed ordering.
		
		By \cref{thm:mainorderedpair}, there exists $c>0$
		such that if an ordered graph $G$ is both $H$-free  and $\overline{H'}$-free, then there is a clique or stable set of size at least $|G|^c$ in $G$.
		Let $T$ be a $Q$-free tournament, take a numbering of its vertex set, and 
		let $P$ be the resulting backedge ordered graph. Then $P$ is $J$-free for every backedge ordered graph of $Q$, because $T$ is $Q$-free;
		and in particular,
		$P$ is both $H$-free and $\overline{H'}$-free. Hence there is a clique or stable set $S$ in $P$ of size at least $|Q|^c$. But then 
		$S$ is a transitive set of $T$. This proves \cref{tourEH}.
	\end{proof}
	\section*{Acknowledgements}
	We are grateful to the anonymous referees for helpful comments and suggestions.
	
	\bibliographystyle{abbrv}

\begin{thebibliography}{30}
		
		\bibitem{MR1832443}
		N.~Alon, J.~Pach, and J.~Solymosi.
		\newblock Ramsey-type theorems with forbidden subgraphs.
		\newblock {\em Combinatorica}, 21(2):155--170, 2001.
		
		\bibitem{bb2022}
		P.~Blanco and M.~Buci{\'c}.
		\newblock Towards the {Erd{\H o}s-Hajnal} conjecture for $P_5$-free graphs,
		\newblock {\em Res. Math. Sci.}, 11(1):2, 2024.
		
		\bibitem{MR3416129}
		N.~Bousquet, A.~Lagoutte, Z.~Li, A.~Parreau, and S.~Thomass\'{e}.
		\newblock Identifying codes in hereditary classes of graphs and {VC}-dimension.
		\newblock {\em SIAM J. Discrete Math.}, 29(4):2047--2064, 2015.
		
		\bibitem{bfp2024}
		M.~Buci{\'c}, J.~Fox, and H.T.~Pham.
		\newblock Equivalence between Erd\H os-Hajnal and polynomial R\"odl and Nikiforov conjectures, {\em preprint},
		\href{https://arxiv.org/abs/2403.08303}{\tt arXiv:2403.08303}, 2024.
		
		\bibitem{density1}
		M.~Buci{\'c}, T.~Nguyen, A.~Scott, and P.~Seymour.
		\newblock Induced subgraph density. I. A loglog step towards {Erd{\H
				o}s--Hajnal},
		{\it Int. Math. Res. Not. IMRN}, 2024(12):9991--1004, 2024.
		
		\bibitem{MR3150572}
		M.~Chudnovsky.
		\newblock The {E}rd{\H o}s-{H}ajnal conjecture -- a survey.
		\newblock {\em J. Graph Theory}, 75(2):178--190, 2014.
		
		\bibitem{pure10}
		M. Chudnovsky, A. Scott, P. Seymour  and S. Spirkl.
		\newblock Pure pairs. X. Tournaments and the strong Erd\H{o}s-Hajnal property.
		\newblock {\em European J. Combinatorics}, 115:103786, 2024.
		
		\bibitem{MR2233847}
		M.~Chudnovsky, N.~Robertson, P.~Seymour, and R.~Thomas.
		\newblock The strong perfect graph theorem.
		\newblock {\em Ann. of Math. (2)}, 164(1):51--229, 2006.
		
		\bibitem{MR2462320}
		M.~Chudnovsky and S.~Safra.
		\newblock The {E}rd{\H o}s-{H}ajnal conjecture for bull-free graphs.
		\newblock {\em J. Combin. Theory Ser. B}, 98(6):1301--1310, 2008.
		
		\bibitem{MR4170220}
		M.~Chudnovsky, A.~Scott, P.~Seymour, and S.~Spirkl.
		\newblock Pure pairs. {I}. {T}rees and linear anticomplete pairs.
		\newblock {\em Adv. Math.}, 375:107396, 2020.
		
		\bibitem{MR4563865}
		M.~Chudnovsky, A.~Scott, P.~Seymour, and S.~Spirkl.
		\newblock Erd{\H o}s-{H}ajnal for graphs with no 5-hole.
		\newblock {\em Proc. Lond. Math. Soc. (3)}, 126(3):997--1014, 2023.
		
		\bibitem{MR599767}
		P.~Erd\H{o}s and A.~Hajnal.
		\newblock On spanned subgraphs of graphs.
		\newblock In {\em Contributions To Graph Theory And Its Applications
			({I}nternat. {C}olloq., {O}berhof, 1977) ({G}erman)}, pages 80--96. Tech.
		Hochschule Ilmenau, Ilmenau, 1977.
		
		\bibitem{MR1031262}
		P.~Erd\H{o}s and A.~Hajnal.
		\newblock Ramsey-type theorems.
		\newblock {\em Discrete Applied Mathematics}, 25:37--52, 1989.
		
		\bibitem{polyP4}
		J.~Fox, T.~Nguyen, A.~Scott, and P.~Seymour.
		\newblock Induced subgraph density. II. Sparse and dense sets in cographs,
		\newblock {\em European J. Combin.}, 124:104075, 2025.
		
		
		\bibitem{MR2455625}
		J.~Fox and B.~Sudakov.
		\newblock Induced {R}amsey-type theorems.
		\newblock {\em Adv. Math.}, 219(6):1771--1800, 2008.
		
		\bibitem{MR1425208}
		A.~Gy{\'a}rf{\'a}s.
		\newblock Reflections on a problem of {E}rd{\H o}s and {H}ajnal.
		\newblock In {\em The Mathematics Of {P}aul {E}rd{\H o}s, {II}}, 
		pages 93--98. Springer, Berlin, 1997.
		
		
		\bibitem{2025thes}
		H.~T. Nguyen.
		\newblock {\em Induced Subgraph Density}.
		\newblock PhD thesis, Princeton University, May 2025.
		
		\bibitem{density6}
		T.~Nguyen, A.~Scott, and P.~Seymour.
		\newblock Induced subgraph density. VI. Bounded VC-dimension, 
		\newblock {\em Adv. Math.}, 482(A):110601, 2025.
		
		\bibitem{density5}
		T.~Nguyen, A.~Scott, and P.~Seymour.
		\newblock Induced subgraph density. V. All paths approach Erd\H{o}s-Hajnal, to appear in {\em Adv. Comb.},
		\href{https://arxiv.org/abs/2307.15032}{\tt arXiv:2307.15032},
		2023.
		
		\bibitem{density7}
		T.~Nguyen, A.~Scott, and P.~Seymour.
		\newblock Induced subgraph density. VII. The five-vertex path, {\em Proc. Lond. Math. Soc.} (3), 132(3):Paper No. e70133, 2026.
		
		\bibitem{MR2271833}
		V.~Nikiforov.
		\newblock Edge distribution of graphs with few copies of a given graph.
		\newblock {\em Combin. Probab. Comput.}, 15(6):895--902, 2006.
		
		\bibitem{MR4273057}
		J.~Pach and I.~Tomon.
		\newblock Erd{\H o}s-{H}ajnal-type results for monotone paths.
		\newblock {\em J. Combin. Theory Ser. B}, 151:21--37, 2021.
		
		\bibitem{MR837962}
		V.~R\"{o}dl.
		\newblock On universality of graphs with uniformly distributed edges.
		\newblock {\em Discrete Math.}, 59(1-2):125--134, 1986.
		
		\bibitem{rodlwinkler}
		V.~R\"{o}dl and P.~Winkler.
		\newblock A Ramsey-type theorem for orderings of a graph. 
		\newblock {\em SIAM J.  Discrete Math.},  2:402--406, 1989.
		
		\bibitem{tomon2020}
		I.~Tomon.
		\newblock {String graphs have the Erd{\H o}s--Hajnal property}.
		\newblock {\em J. Eur. Math. Soc.}, 26(1):275--287, 2024.
		
	\end{thebibliography}

\end{document}